\documentclass{amsart}
\usepackage[utf8]{inputenc}

\usepackage{amsfonts}
\usepackage{amsmath}
\usepackage{amssymb}
\usepackage{amsthm}
\usepackage[margin = 1in]{geometry}
\usepackage{todonotes}
\usepackage{url}
\usepackage{cancel} 
\usepackage{forest}
\usepackage{lscape}
\usepackage{hyperref}
\usepackage{enumitem}   
\usepackage{stmaryrd } 
\usepackage{MnSymbol}  
\usepackage{thmtools}   
\usepackage{thm-restate}  

\usepackage{biblatex}
\newtheorem{Thm}{Theorem}[section]
\newtheorem{Lem}[Thm]{Lemma}
\newtheorem{Prop}[Thm]{Proposition}
\newtheorem{Cor}[Thm]{Corollary}

\theoremstyle{definition}
\newtheorem{Def}[Thm]{Definition}
\newtheorem{remark}[Thm]{Remark}
\newtheorem{Ex}[Thm]{Example}

\numberwithin{equation}{section}

\newcommand{\NN}{{\mathbb{N}}}
\newcommand{\ZZ}{{\mathbb{Z}}}

\newcommand{\RR}{{\mathbb{R}}}

\newcommand{\set}[1]{\left\{ #1 \right\}}
\DeclareMathOperator\supp{supp}
\DeclareMathOperator\rev{rev}
\DeclareMathOperator\comp{comp}
\DeclareMathOperator\revcomp{revcomp}
\DeclareMathOperator\desc{desc}
\DeclareMathOperator\seq{seq}
\newcommand{\cover}{\leqslantdot}
\newcommand{\Flow}{\textbf{Flow}}
\newcommand{\nset}[1]{ [ #1 ]}
\DeclareMathOperator{\composition}{\models}
\DeclareMathOperator{\leqa}{\leq_{\text{a}}}
\DeclareMathOperator{\leqb}{\leq_{\text{b}}}
\DeclareMathOperator{\leqc}{\leq_{\text{c}}}
\DeclareMathOperator{\geqa}{\geq_{\text{a}}}

\DeclareMathOperator{\geqc}{\geq_{\text{c}}}
\DeclareMathOperator{\covera}{\leqdot_{\text{a}}}
\DeclareMathOperator{\coverb}{\leqdot_{\text{b}}}

\DeclareMathOperator{\Prim}{\mathcal{P}}  
\newcommand{\defn}[1]{{\color{blue} \it {#1}}} 

\title[On the $f$-vectors of flow polytopes for the complete graph]{On the $f$-vectors of flow polytopes \\ for the complete graph}
\author[W. Dugan]{William T. Dugan}
\address[W. Dugan]{Department of Mathematics and Statistics, University of Massachusetts Amherst, Amherst, MA, U.S.A}
\email{wdugan@umass.edu}
\urladdr{https://sites.google.com/view/william-dugan/home}
\thanks{This project was partially supported by NSF grants DMS-1855536 and DMS-2154019.}

\date{\today}

\begin{document}

\begin{abstract}
    The Chan-Robbins-Yuen polytope ($CRY_n$) of order $n$ is a face of the Birkhoff polytope of doubly stochastic matrices that is also a flow polytope of the directed complete graph $K_{n+1}$ with netflow $(1,0,0, \ldots , 0, -1)$. The volume and lattice points of this polytope have been actively studied, however its face structure has received less attention. We give generating functions and explicit formulas for computing the $f$-vector by using Hille's (2003) result bijecting faces of a flow polytope to certain graphs, as well as Andresen-Kjeldsen's (1976) result that enumerates certain subgraphs of the directed complete graph. We extend our results to flow polytopes of the complete graph having arbitrary (non-negative) netflow vectors and recover the $f$-vector of the Tesler polytope of M\'esz\'aros--Morales--Rhoades (2017).
\end{abstract}

\maketitle

\section{Introduction}

The Chan-Robbins-Yuen polytope ($CRY_n$) of order $n$ is defined as the convex hull of $n$-by-$n$ permutation matrices $\pi$ for which $\pi_{i,j} = 0$ for $j \geq i+2$ \cite{Chan_Robbins_Yuen_2000}. This polytope has been the object of much interest in the research community, as it possesses many interesting traits. For example, Zeilberger proved in \cite{zeilberger_proof_1998} using a variation of the Morris constant term identity that $CRY_n$ has normalized volume equal to the product of the first $n-2$ Catalan numbers. A second algebraic proof was provided in \cite{Baldoni_Vernge_2001}, though a purely combinatorial proof of this fact remains elusive.  $CRY_n$ is also a face of the Birkhoff polytope of doubly stochastic matrices  having dimension $\binom{n}{2}$ and $2^{n-1}$ vertices \cite{Chan_Robbins_Yuen_2000}.

Furthermore, $CRY_n$ is an example of a more general family of polytopes, namely those which are flow polytopes of the complete (transitively directed) graph $K_{n+1}$ on vertex set $\set{v_1, \ldots , v_{n+1}}$, which include the family of Tesler polytopes \cite{Tesler_polytopes_2015}. We denote this family by $\Flow_n({\bf a})$  and the corresponding $f$-vectors and $f$-polynomials as $f^{(n)}({\bf a})$ and $f^{(n)}({\bf a}; x)$ respectively (see Definition~\ref{def::complete_graph_flow_polytopes}).

In particular, $CRY_n$ is realized as an instance of $\Flow_n({\bf a})$ by setting ${\bf a} = (1, 0, \ldots , 0)$. $\Flow_n({\bf a})$ has also been  studied by M\'esz\'aros--Morales--Rhoades \cite{Tesler_polytopes_2015} in the context of \textit{Tesler polytopes}, i.e. flow polytopes on the complete graph with positive netflows, in which they show that the case of all $a_i > 0$, such as ${\bf a} = (1,1,\ldots, 1)$, is combinatorially equivalent to a product of simplices $\Delta_n \times \Delta_{n-1} \times \ldots \times \Delta_1$. This was later generalized to other graphs by M\'esz\'aros--Simpson--Wellner \cite{Meszaros2017FlowPO}. Part of the difficulty in obtaining the $f$-vector of $\Flow_n({\bf a})$ for more general ${\bf a}$ arises from the fact that $\Flow_n(1,1, \ldots, 1)$ is simple, whereas general instances of $\Flow_n({\bf a})$ (including the case of $CRY_n$) are not.
 
In this manuscript, we give an explicit formula for the $f$-vector of $\Flow_n({\bf a})$ for any non-negative ${\bf a}$ as a sum over certain compositions. Namely, given a netflow vector ${\bf a}$, let $\revcomp({\bf a})$ be the composition obtained by reading the entries of ${\bf a}$ from right to left, inductively creating blocks whenever a new nonzero entry is encountered, and recording the tuple of sizes coming from the list of blocks (see Definition~\ref{def::revcomp} and  Example~\ref{ex::revcomp_ex}). Furthermore, let $\leqa$ be the partial order of refinement on compositions, and let $\ell(\alpha)$ be the number of parts of composition $\alpha$. 

\begin{restatable*}{Thm}{MainResult}
    \label{thm::main_result}
Given a netflow vector $({\bf a}, -|{\bf a}|) = (a_1, \ldots a_n, -\sum_{i = 1}^n a_i)$ with $a_i \in \mathbb{N}$, let $\alpha$ be the integer composition of $n$ given by $\alpha = \revcomp({\bf a})$. Define $Q(x_0,x_1, \ldots, x_n)$ to be the Laurent polynomial:
$$Q_{\alpha}(x_0,x_1, \ldots, x_n) :=  \frac{1}{x_0} + \frac{1}{x_0^n}\sum_{\beta \geqa \alpha}  (-1)^{\ell(\alpha) - \ell(\beta)} \pi_{\ell(\beta)}(x_0) {\bf x}^{\beta-{\bf 1}},$$
where $\pi_{n}(x_0):=x_0^{n}[n]_{x_0+1}! = \prod_{i = 1}^{n}((x_0+1)^i - 1)$.
Then the $f$-polynomial of $\Flow_n({\bf a})$ is given by:
     \begin{equation}
     \label{eq::f_vec_main_result}
        f^{(n)}({\bf a}; x)  = Q_{\alpha}(x, (x+1)^1 - (x+1), (x+1)^2 - (x+1), \ldots , (x+1)^n - (x+1) ).
     \end{equation}
     
\end{restatable*}

The reader may notice that Equation~\eqref{eq::f_vec_main_result} looks almost like an evaluation of a quasisymmetric function. We will discuss this viewpoint in Section~\ref{subsect::formulas_as_evaluations_of_sums_of_quasisymmetric_polynomials}.

Note that in the case of $a_i > 0$ for all $i$, we recover the results of \cite[Thm 1.7]{Tesler_polytopes_2015} that $f({\bf a}; x) = [n]_{x+1}!$, a consequence of $\Flow_n(1,1,\ldots , 1)$ being combinatorially equivalent to a product of simplices $\Delta_n \times \Delta_{n-1} \times \ldots \times \Delta_1$ as referenced above. In the case that ${\bf a} = (1, 0, \ldots , 0)$, we obtain a succinct formula for the previously-unknown $f$-vector of $CRY_n$ as a sum over complete homogeneous symmetric functions $h_m({\bf x}) := \sum_{1 \leq i_1 \leq \ldots \leq i_n}x_{i_1}\cdots x_{i_n}$.

\begin{restatable*}{Cor}{CorSpecializeCRY}
    \label{cor::CRY_f_vec_formula}
     Let $f^{(n)}(x)$ be the  $f$-vector of $CRY_n = \Flow_n(1,0, \ldots, 0)$ written as a Laurent polynomial. Then for all $n \geq 1$:
     \begin{equation}
        f^{(n)}(x) = \frac{1}{x} +  \frac{1}{x^n}\sum_{m = 0}^{n-2}(-1)^m(1+x)^m\pi_{n-m}(x)\cdot h_m(x, (x+1)^2 - 1, \ldots , (x+1)^{n-m - 1} -1)
    \end{equation}
\end{restatable*}

This is a direct generalization of a theorem due to Andresen--Kjeldsen \cite[Prop. 3.3]{andresen_kjeldsen_1976} (which is recovered by setting $x = 1$) enumerating certain subgraphs of $K_{n+1}$. In their paper, the authors of \cite{andresen_kjeldsen_1976} study two families of subgraphs originating from their prior work in automata theory: namely the set $\Omega_n$ of subgraphs $H$ of $K_{n+1}$ for which every vertex of $H$ lies along some path from the first vertex to the last vertex (see Equation~\eqref{eq::Omega_n_def}) and the set $\Omega_n'$ of those subgraphs which satisfy this property and also contain every vertex in the same connected component (see Equation~\eqref{eq::Omega_n_prime_def}). They give multiple formulas for the cardinalities $|\Omega_n|$ and $|\Omega_n'|$ (see for example equations~\eqref{eq::Andresen_Kjeldsen_psi_n} and \eqref{eq::Andresen_Kjeldsen_xi_n}) which our equations recover in the limit $x \rightarrow 1$. Additionally, they demonstrate how one may actually recover $|\Omega_n|$ from $|\Omega_n'|$ (and vice versa), as shown in \cite[eq. 1]{andresen_kjeldsen_1976}, which is a special case of our Corollary~\ref{cor::CRY_binomial_f_vec_f_primitive_vec_relationship}.

The connection between Corollary~\ref{cor::CRY_f_vec_formula} and Equation~\eqref{eq::Andresen_Kjeldsen_psi_n} is made explicit  via a  powerful theorem of Hille \cite{Hille_2003}, which relates faces of a flow polytope for some graph $G$ with certain subgraphs of $G$. Under the correspondence, dimension of a face of the flow polytope corresponds to the first Betti number of the corresponding subgraph (see Theorem~\ref{thm::Hille}).

In this way, we see that the $f$-polynomial of $CRY_n$ is exactly a generating function over $\Omega_n$, where variable $x$ keeps track of the first Betti number of $H \in \Omega_n$. This connection leads us to define a new notion of \defn{primitive $f$-polynomial} (respectively \defn{primitive $f$-vector}) of $\Flow_n({\bf a})$, denoted \defn{$\widetilde{f}^{(n)}({\bf a}; x)$} (respectively \defn{$\widetilde{f}^{(n)}({\bf a})$}), which is a generating function for the first Betti number of each graph over the subset of graphs which are primitive (see Definition~\ref{def::primitive_f_vector}).

Later in the text, we describe closed-form expressions for the primitive $f$-polynomial of $\Flow_n({\bf a})$ (Lemma~\ref{lem::primitive_f_vec_arbitrary_net_flow_sum_over_sets} and  Lemma~\ref{lem::primitive_f_vec_as_quasi_sym}) and describe a relationship between $f^{(n)}({\bf a} ; x)$ and $\widetilde{f}^{(n)}({\bf a} ; x)$ for arbitrary (non-negative) ${\bf a}$ (Lemma~\ref{lem::generalized_psi_xi_relationship}), as well as the special case of $CRY_n$ (Corollary~\ref{cor::CRY_binomial_f_vec_f_primitive_vec_relationship}). 

A second, special relationship exists between the $f$-vector and primitive $f$-vector in the case of $CRY_n$, and specializes to  \cite[Prop. 4.1 ]{andresen_kjeldsen_1976} of Andresen--Kjeldsen by setting $x = 1$:

\begin{restatable*}{Thm}{fVectorPrimitiveRelationship}
    \label{thm::f_vector_product_result}
    For all $n \geq 1$, the $f$-polynomial and primitive $f$-polynomial of $CRY_n$ are related as:
    \begin{equation}
    \label{eq::f_vector_primitive_f_vec_product_result}
        xf^{(n)}(x) = (1+x)^{n-1}\widetilde{f}^{(n)}(x).
    \end{equation}
\end{restatable*}

Finally, we remark that Jel\'inek \cite{Jelinek_2011} observed that $\Omega_n'$ is in fact in bijection with the set of  \defn{primitive Fishburn matrices} (upper triangular, $0$-$1$ matrices such that no row nor column is the zero vector) , and consequently is related to the enumeration of interval orders \cite{Bousquet_Melou_2010}, by interpreting $H \in \Omega_n'$ as the upper-triangular matrix determined by its edges. 
As discussed in \cite{hwang_jin_schlosser_2020}, the bijections continue, as the more general notion of Fishburn matrices are in bijection with  Stoimenow matchings, ascent sequences, and more \cite{Bousquet_Melou_2010, Stoimenow_1998}. See \cite{hwang_jin_schlosser_2020, Jelinek_2011} for a more comprehensive list of related combinatorial objects.

Either from Corollary~\ref{cor::CRY_f_vec_formula} or from a multivariate generating function of Fishburn matrices due to Jel\'inek \cite[Thm. 2.1]{Jelinek_2011} one obtains the following nice generating function for $d$-dimensional faces of $CRY_n$ for varying $d$ and $n$.

\begin{restatable*}{Cor}{CorGeneratingFunction}
\label{cor::generating_function_for_CRY_faces}
The number of $d$-dimensional faces of $CRY_n$ is given by the coefficient $f^{(n)}_d = [t^nx^{d}]F(t,x)$, where $F(t,x)$ is defined by:
    \begin{equation}
        F(t,x) := \frac{1}{x-xt} + \sum_{n = 0}^{\infty} t^nx^{-n} \prod_{i = 1}^{n}\frac{(1 + x)^i - 1}{1 + ((1+x)^i - 1 - x)tx^{-1}}.
    \end{equation}
\end{restatable*}

The rest of this paper is organized as follows: In Section~\ref{sect::Background} we provide a more comprehensive overview of the background material needed in order to understand the paper's main results. In Section~\ref{sect::main_results}, we derive our main result Theorem~\ref{thm::main_result} as well as results for general primitive $f$-polynomials (Lemma~\ref{lem::primitive_f_vec_arbitrary_net_flow_sum_over_sets} and Lemma~\ref{lem::primitive_f_vec_as_quasi_sym}) needed in the proof.  In Section~\ref{sect::formulas_for_CRY_n} we specialize our results to $CRY_n$. Finally, we conclude in Section~\ref{sect::direct_enumerative_results_and_recurrences}, by describing some of the direct enumerative results for faces of $\Flow_n({\bf a})$ that follow from Theorem~\ref{thm::Hille}, as well as recurrences that can be obtained from such. In particular, we give formulas for the number of vertices of $\Flow_n({\bf a})$ (Proposition~\ref{prop::number_of_vertices}) and the number of faces of low codimension (Theorem~\ref{cor::CRY_faces_low_codim}) that are more computationally efficient than computing the entire $f$-vector.

\section{Background}

In this section, we provide further background in order to make the paper as self-contained as possible. Readers already familiar with this material may wish to skip to Section~\ref{sect::main_results}.

\label{sect::Background}

\subsection{Posets of compositions}

\label{subsect::posets_of_compositions}

There exist many different partial orders involving integer compositions, and in this text we will require a few. To fix notation, let $\mathcal{C}_n$ denote the set of integer compositions of $n$, and $\mathcal{C} := \bigcup_{n \geq 0} \mathcal{C}_n$ the set of compositions of all $n \in \NN$. We will write $\alpha \composition n$ if $\alpha \in \mathcal{C}_n$. 

Compositions also come with two different notions of size. If $\alpha = (\alpha_1, \alpha_2, \ldots, \alpha_r) \composition n$, we will write $\ell(\alpha) := r$ for the number of parts of $\alpha$ and $|\alpha| := n$ for the sum of its parts.  

The first partial order of interest is one on $C_n$, namely that of \textit{reverse refinement} used frequently in the theory of quasisymmetric functions. If we denote this poset by $(\mathcal{C}_n, \leqa)$ and we have two integer compositions $\alpha = (\alpha_1, \alpha_2, \ldots, \alpha_r) \composition n$ and $\beta = (\beta_1, \ldots \beta_s) \composition n$ then $\alpha \leqa \beta$ if and only if each of the $\beta_i$'s are compositions of the $\alpha_j$'s. Said another way, the cover relations in  $(\mathcal{C}_n, \leqa)$ have the form \cite{Petersen_2005}:
$$(\alpha_1, \alpha_2, \ldots, \alpha_i + \alpha_{i+1}, \ldots  \alpha_r) \covera (\alpha_1, \alpha_2, \ldots, \alpha_i,  \alpha_{i+1}, \ldots  \alpha_r)$$
We observe that these posets are easily seen to be isomorphic to Boolean lattices.

The second partial order we need is one on $\mathcal{C}$ defined by Bj\"orner and Stanley \cite{Bjorner_Stanley_2005}. If we denote this poset by $(\mathcal{C}, \leqb)$, recall that $\leqb$ is defined via two types of cover relations (which we will denote $\coverb_1$ and $\coverb_2$). For compositions $\alpha$ and $\beta$, these are described by (\cite{Bjorner_Stanley_2005}):

\medskip

\begin{itemize}
    \item (type I cover) $\alpha$ is covered by  $\beta$ if $\beta$ can be obtained from $\alpha$ by adding $1$ to a part, and
    \item (type II cover) $\alpha$ is covered by  $\beta$ if $\beta$ can be obtained from $\alpha$ by adding $1$ to a part and then splitting this part into two parts. 
\end{itemize}
     
This poset (minus the red lassos) is depicted in Figure~\ref{fig::composition_poset} (left). Finally, we introduce a third poset relevant to the computations in Section~\ref{sect::main_results}. 

\begin{Def}
    Define $(\mathcal{C}, \leqc)$ to be the coarsening of $(\mathcal{C}, \leqb)$ by  taking only the transitive closure of the type I cover above. See Figure \ref{fig::composition_poset} (right) for an example.
\end{Def}

We will be most interested in downsets of $(\mathcal{C}, \leqc)$, which are easily seen to be isomorphic to products of chains. 

\begin{figure}
    \centering
    \includegraphics[width = \textwidth]{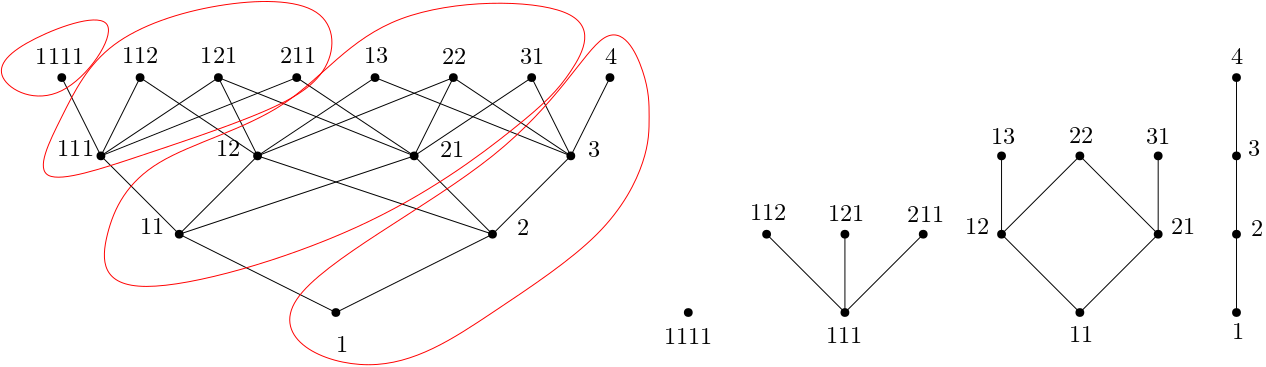}
    \caption{The composition poset $(\mathcal{C}, \leqb)$ of \cite{Bjorner_Stanley_2005} with lassos indicating the downsets determined by $\cover_b1$ (left) and the coarsening $(\mathcal{C}, \leqc)$ (right).}
    \label{fig::composition_poset}
\end{figure}

It is well-known that elements of $\mathcal{C}_n$ are in bijection with $0$-$1$ sequences of length $n-1$. For formulas to be stated more simply in the following sections, we will use a slight variant of this bijection described as follows.

\begin{Def}
\label{def::revcomp}
    For a subset $S \subseteq [n]$, let the \defn{reverse} of $S$, denoted $\rev(S)$,  be defined as $\rev(S) = \{n + 1 - i \, | \, i \in S \}$.
For a natural number vector ${\bf a} := (a_1, \ldots, a_n)$ such that $a_1 \neq 0$, we define the \defn{reverse composition}, $\revcomp({\bf a})$,  as the composition corresponding to the set $\rev(\supp({a_2, \ldots, a_n}))$.    Computationally, $\revcomp({\bf a})$ may be obtained quickly by reading the entries of ${\bf a}$ from right to left, inductively creating blocks whenever a new nonzero entry is encountered, and recording the tuple of sizes coming from the list of blocks. \footnote{The reason for choosing this map as opposed to the usual map between $0$-$1$ vectors and integer compositions arises from the application of this map to the enumeration of faces of flow polytopes in Section~\ref{sect::main_results}. In particular, one could instead choose to use the usual map between $0$-$1$ vectors and compositions at the expense of labelling the vertices of $K_{n+1}$ from right to left. }
\end{Def}

\begin{Ex}
\label{ex::revcomp_ex}
  If ${\bf a} = (1,1,0,0,1,0,1,0)$, then when we read ${\bf a}$ right to left, we first encounter the block $(0,1)$, followed by $(0,1)$, followed by $(0,0,1)$ and finally $(1)$. The reverse composition of ${\bf a}$ is then obtained by writing down the sizes of these blocks, hence $\revcomp({\bf a}) = (2,2,3,1)$.  For a non $0$-$1$ vector, we may first replace every nonzero entry with a $1$ and then perform the same procedure described here.
\end{Ex}

\begin{remark}
    The condition that $a_1 \neq 0 $ in Definition~\ref{def::revcomp} is imposed because we can always make this assumption on the netflow vector for a flow polytope (see Section~\ref{subsect::flow_polytopes} for definitions and notations related to flow polytopes). Indeed, if $a_1 = 0$, we may simply consider the flow polytope $\mathcal{F}_{G - \{v_1\} }(a_2, \ldots a_n)$ and all of the relevant combinatorics will be preserved. Additionally, note that such an assumption ensures that $\revcomp$ induces a bijection from the set of length-$n$ bit strings $(a_1, \ldots , a_n)$ such that $a_1 \neq 0$ and integer compositions of length $n$ (instead of length $n+1$). 
\end{remark}

\begin{remark}
\label{rem::composition_posets_as_bitstrings}
    Denote by $\mathcal{B}_n$ the set of all $0$-$1$ vectors of length $n$ such that the first entry is not $0$, and denote by $\mathcal{B} := \bigcup_{n \geq 0} \mathcal{B}_n $ the collection of all of these vectors together. We observe that each of the partial orders on compositions mentioned above translates to a partial order on $\mathcal{B}_n$ (respectively $\mathcal{B}$). We describe this translation here (reusing the symbols $\leqa, \leqb, $ and $\leqc$ by abuse of notation).
    \begin{itemize}
        \item The partial order $\leqa$ interpreted in $\mathcal{B}_n$ becomes ${\bf a} \leqa {\bf b}$ if and only if ${\bf b}$ may be obtained from ${\bf a}$ by flipping $0$'s to $1$'s.
        \item The partial order $\leqb$ interpreted in $\mathcal{B}$ becomes 
    ${\bf a} \leqb {\bf b}$ if and only if ${\bf a}$ may be obtained from ${\bf b}$ by deleting entries.
        \item  Finally, the partial order $\leqc$ in $\mathcal{B}$ becomes ${\bf a} \leqc {\bf b}$ if and only if ${\bf a}$ can be obtained from ${\bf b}$ be deleting $0$-entries. 
    \end{itemize}

    The two notions of size on compositions also translate to notions of size on $0$-$1$ vectors. However, one must take care as the two notions of size seem to flip under the correspondence. For ${\bf a}$ a $0$-$1$ vector, we denote by $\ell({\bf a})$ the length of ${\bf a}$ and by $|{\bf a}|$ the sum of the entries of ${\bf a}$ (i.e. the number of $1$'s). Then under the bijection described above, $\ell({\bf a}) = |\revcomp({\bf a})|$ and  $|{\bf a}| = \ell(\revcomp({\bf a}))$. 
    
     
\end{remark}

 \begin{remark}
 \label{rem::signature}
     Recall that the \defn{signature} of a binary vector ${\bf v} = (\underbrace{1, 0, \ldots 0}_{k_1}, \underbrace{1, 0, \ldots , 0}_{k_2}, \ldots \underbrace{1, 0, \ldots 0}_{k_r})$ is the tuple $\defn{sgn({\bf v})} := (k_1, k_2, \ldots, k_r)$ (see also \cite{PS_vertices}). We note that $\revcomp({\bf a})$ may also be thought of as the signature of the vector ${\bf b}$, where ${\bf b}$ is the vector obtained by first reversing ${\bf a}$ and then flipping every bit. 
 \end{remark}

\subsection{Flow Polytopes}

\label{subsect::flow_polytopes}

The main class of objects in this paper is that of a flow polytope, a polytope determined by the data of a directed, acyclic graph having an integer weight on each of its vertices. Points of the polytope are tuples of real numbers that correspond to valid flows  over the edges of the graph that conserve the netflow determined by each vertex's weight. We make this precise as follows.

\begin{Def}[Flow] 
    Let $G = (V, E)$ be a directed, acyclic graph on the vertex set $V = [n+1]$, and let ${\bf a}' = (a_1, \ldots , a_n, a_{n+1}) \in \ZZ^{n+1}$ be a vector of weights on the vertex set $V$ (with vertex $i$ having weight $a_i$) such that $a_{n+1} = -\sum_{i = 1}^{n}a_i$. Then ${\bf a}'$ is called a \defn{netflow vector} and a function $f: E \rightarrow \RR$ is an \defn{${\bf a}'${-flow}} (or simply a \defn{flow}, if the netflow vector ${\bf a}'$ is understood from context) if for each $i \in [n]$:
    \begin{equation}
        \label{eq::conservation_of_flow}
        \sum_{e = (i,j) \in E} f(e) - \sum_{e = (k,i) \in E} f(e) = a_i
    \end{equation}
    In other words, the difference between the incoming flow at vertex $i$ and outgoing flow at vertex $i$ is exactly $a_i$. The reader may be familiar with Equation~\eqref{eq::conservation_of_flow} as \textit{Kirchhoff's Law} when the netflow at a vertex is $0$. Indeed, one may think of Equation~\eqref{eq::conservation_of_flow} as conservation of flow on a network after interpreting a positive $a_i$ as being the amount of flow entering the network at node $v_i$ and a negative $a_i$ as the being the amount of flow leaving at node $v_i$.
\end{Def}

\begin{Def}[Flow polytope]
    Let $G$ be a directed, acyclic graph and let ${\bf a}'$ be a netflow vector. Then the \defn{flow polytope} determined by $G$ and ${\bf a}'$ is the convex set:
    \begin{equation}
        \mathcal{F}_G({\bf a}') := \{ {\bf x}\in \RR^{|E|} \, | \, {\bf x} = (f(e_1), \ldots , f(e_{|E|})) )  \text{ for  an ${\bf a}'$-flow } f\}
    \end{equation}
\end{Def}

Note in particular that ${\bf a}' \in \ZZ^{n+1}$ implies that $\mathcal{F}_G({\bf a}')$ is integral, hence all polytopes in this paper will be integral. Moreover, isomorphism in the category of flow polytopes is given by integral equivalence; that is, affine transformations that preserve the integer lattice. We will denote integral equivalence by the symbol $\equiv$.

One approach to studying flow polytopes is to study families of polytopes arising from a known family of graphs. In this paper, we focus on the family of flow polytopes $\{ \mathcal{F}_{K_{n+1}}({\bf a}') \}_{n \geq 1}$ arising from the complete graphs $K_{n+1}$. We fix notation for this family here.

\begin{Def}
\label{def::complete_graph_flow_polytopes}
    For $n \in \NN$ and ${\bf a} \in \NN^n$, we denote the flow polytope $\mathcal{F}_{K_{n+1}}({\bf a}, -\sum_{i=1}^{n} a_i)$ as $\defn{\Flow_n({\bf a})}$.  We will denote the $f$-vector of $\Flow_n({\bf a})$ by $\defn{f^{(n)}({\bf a})}$ (or \defn{$f^{(n)}({\bf a}; x)$} if written as a Laurent polynomial, where the coefficient of $x^i$ gives the number of $i$-dimensional faces for $i \geq -1$).
\end{Def}

Some of our results (c.f. Theorem~\ref{thm::f_vector_product_result}) specifically require us to include the $\frac{1}{x}$ term in $f^{(n)}({\bf a}; x)$ corresponding to the empty face. As this is the only non-polynomial term of $f^{(n)}({\bf a}; x)$, however, we will abuse terminology and often refer to $f^{(n)}({\bf a}; x)$ as the \defn{$f$-polynomial} of $\Flow_n({\bf a})$.  The first few values of $f^{(n)}(1,0,\ldots, 0)$ are included in Table~\ref{tab::f_vectors_of_CRY_n}.

\begin{table}[]
    \centering
    \begin{tabular}{c p{15cm}}
    \hline
    $n$ & $f$-vector of $CRY_n$\\
    \hline
      $-1$   &  $(1)$\\
       $0$  &  $(1)$ \\
       $1$   &  $(1, 1)$\\
       $2$  &  $(1, 2, 1)$ \\
       $3$   &  $(1, 4, 6, 4, 1)$\\
       $4$  &  $(1, 8, 26, 45, 45, 26, 8, 1)$ \\
       $5$   &  $(1, 16, 98, 327, 681, 944, 897, 588, 262, 76, 13, 1)$\\ 
       $6$  &  $(1, 32, 342, 1943, 6982, 17326, 31236, 42198, 43521, 34601, 21249, 10020, 3571, 933, 169, 19, 1)$ \\
       $7$   &  $(1, 64, 1138, 10275, 58093, 228396, 664200, 1486921, 2633161, 3759650, 4386239, 4218971, 3363558, $ \newline $2227042, 1222927, 554147, 205256, 61206, 14351, 2550, 323, 26, 1)$\\
       $8$  &  $(1, 128, 3670, 50403, 424214, 2468235, 10653629, 35711651, 95967645, 211567734,  389268482, 605593465, $ \newline $ 804533944, 919531124, 909049826, 780149435, 582376682, 378321185, 213630918, 104570683, 44165758, $ \newline $ 15985336, 4910781, 1263620, 267378, 45321, 5918, 559, 34, 1)$\\
       \hline
    \end{tabular}
    \caption{The first few $f$-vectors of $CRY_n.$}
    \label{tab::f_vectors_of_CRY_n}
\end{table}

\begin{remark}
    The family of flow polytopes $\Flow_n({\bf a})$ have gone by a few different names in the literature, including \textit{Tesler polytopes} \cite{Tesler_polytopes_2015} and \textit{Baldoni-Vergne polytopes} \cite{meszaros_morales_ehrhart_2019}. More recently, \textit{Tesler polytopes} refer to the case of $\Flow_n({\bf a})$ in which all $a_i > 0$ \cite{ meszaros_morales_ehrhart_2019, Meszaros2017FlowPO}. To avoid any ambiguity we will simply refer to them throughout as $\Flow_n({\bf a})$ in keeping with the notation of \cite{Tesler_polytopes_2015}.
\end{remark}

\subsection{Faces of Flow Polytopes}

One of the main goals of this manuscript is to find a formula for $f^{(n)}({\bf a})$ depending only on the input vector ${\bf a}$. To achieve this goal, we will make frequent use of the following theorem relating faces of a flow polytope $\mathcal{F}_{G}({\bf a})$ with certain subgraphs of $G$. Here a subgraph $H \subseteq G$ is \textbf{${\bf a}$-valid} if $H$ is the support of an ${\bf a}$-flow on $G$, and the \textit{first Betti number of $H$} is $\beta_1(H) := |E(H)| - |V(H)| + c(H)$, where $c(H)$ is the number of connected components of $H$. See also \cite{PS_vertices}. 

\begin{Thm}[{\cite{Hille_2003}}]
\label{thm::Hille}
    Let $\mathcal{F}_G({\bf a})$ be a flow polytope such that $a_i \geq 0$ for all $i$. Then for $d \geq 0$, the $d$-dimensional faces of $\mathcal{F}_G({\bf a})$ are in one-to-one correspondence with subgraphs $H \subseteq G$ such that $H$ is ${\bf a}$-valid and $\beta_1(H) = d$. The empty face of $\mathcal{F}_G({\bf a})$ corresponds to the empty subgraph of $G$.
\end{Thm}

An illustration of Theorem~\ref{thm::Hille} may be seen in Figure~\ref{fig::faces_of_CRY_3}.

\subsection{Primitive graphs and $f$-vectors}

Andresen and Kjeldsen \cite{andresen_kjeldsen_1976} were the first to study the families $\Omega_n$  and $\Omega_n'$ of subgraphs of the complete graph defined as follows:
\begin{equation}
\label{eq::Omega_n_def}
    \Omega_n := \set{H \subseteq K_{n+1} \, | \, \text{every $v \in V(H)$ lies along a direct path from $v_1$ to $v_{n+1}$}}
\end{equation}
\begin{equation}
\label{eq::Omega_n_prime_def}
    \Omega_n' := \{H \in \Omega_n \, | \, V(H) = \set{v_1, \ldots, v_{n+1}} \text{ and $H$ is connected} \}.
\end{equation}

In the language of flow polytopes, $\Omega_n$ is exactly the set of subgraphs of $K_{n+1}$ that correspond to faces of $CRY_n$ under the bijection of Theorem~\ref{thm::Hille}, while $\Omega_n'$ is the subset of $\Omega_n$ of graphs with all vertices contained in a single connected component. They found numerous formulas for the cardinalities for these sets, for example\footnote{In \cite{andresen_kjeldsen_1976}, the authors refer to these cardinalities as $\psi(n) := |\Omega_n|$ and $\xi(n) := |\Omega_n'|$.}: 
\begin{equation}
\label{eq::Andresen_Kjeldsen_psi_n}
         |\Omega_n| = \sum_{m = 0}^{n-2}(-2)^m\pi_{n-m}\cdot h_m(2^1 - 1, 2^2 - 1, \ldots , 2^{n-m - 1} -1)
\end{equation}
\begin{equation}
\label{eq::Andresen_Kjeldsen_xi_n}
         |\Omega_n'| = \sum_{m = 0}^{n-1}(-1)^m\pi_{n-m}\cdot h_m(2^1 - 1, 2^2 - 1, \ldots , 2^{n-m} -1),
\end{equation}
where $\pi_n := \prod_{i = 1}^n (2^i - 1)$. These cardinalities give rise to sequences \cite[\href{https://oeis.org/A005016}{A005016}]{oeis} and \cite[\href{https://oeis.org/A005321}{A005321}]{oeis}, respectively.

In particular, by Theorem~\ref{thm::Hille} $|\Omega_n|$ enumerates the total number of faces of $CRY_n$ and $f^{(n)}(1,0,\ldots, 0 ; x)$ is a generating function over the set $\Omega_n$ keeping track of the first Betti number of each graph.

Andresen and Kjeldsen referred to elements of $\Omega_n'$ as \textit{primitive} graphs. We generalize this notion and say that a subgraph $H$ of $K_{n+1}$ is \defn{primitive} if is connected and has vertex set $[n+1]$. In the same way that $f^{(n)}({\bf a}; x)$ is a generating function over the set of faces of $\Flow_n({\bf a})$ keeping track of the first Betti number of graphs, we are interested in the analogous generating function over the set of faces of $\Flow_n({\bf a})$ that correspond to primitive graphs.

\begin{Def}
\label{def::primitive_f_vector}
    The \defn{primitive $f$-polynomial} of $\Flow_n({\bf a})$, denoted $\widetilde{f}^{(n)}({\bf a}; x)$ (or as $\widetilde{f}^{(n)}({\bf a})$ if written as a vector) is a generating function over the set of ${\bf a}$-valid subgraphs of $K_{n+1}$ that are primitive (use the entire vertex set) keeping track of the first Betti number.
\end{Def}

Note that it follows immediately from the definition that $\widetilde{f}^{(n)}(1,0,\ldots, 0;x)|_{x = 1} = |\Omega_n'|$ in the same way that $f^{(n)}(1,0,\ldots, 0;x)|_{x = 1} = |\Omega_n|$ from Theorem~\ref{thm::Hille}. The first few values of $\widetilde{f}^{(n)}(1,0,\ldots , 0)$ are included in Table~\ref{tab::primitive_f_vectors_of_CRY_n}. An example computation of $f^{3}(1,0,0;x)$ and $\widetilde{f}^{3}(1,0,0;x)$ may be seen in Figure~\ref{fig::faces_of_CRY_3}.

\begin{table}[]
    \centering
    \begin{tabular}{c p{15cm}}
    \hline \vspace*{-3mm}\\
    $n$ & $\widetilde{f}$ of $CRY_n$\\
    \hline
      $0$   &  $(1)$\\
       $1$  &  $(0, 1)$ \\
       $2$   &  $(0, 1, 1)$\\
       $3$  &  $(0, 1, 4,4,1)$ \\
       $4$   &  $(0, 1,11,33,42,26,8,1)$\\
       $5$  &  $(0, 1, 26, 171, 507,840, 865, 584, 262, 76, 13,1)$ \\
       $6$   &  $(0, 1, 57, 718, 4017, 12866, 26831, 39268, 42211, 34221, 21184, 10015, 3571, 933, 169, 19, 1)$\\
       $7$  &  $(0, 1, 120, 2682, 25531, 138080, 490079, 1242533, 2375965, 3553184, 4258940, 4158866, 3342132, 2221444, $ \newline $ 1221913, 554033, 205250, 61206, 14351, 2550, 323, 26, 1)$ \\
       $8$ & $(0, 1, 247, 9327, 141904, 1201179, 6629070, 26168817, 78440289, 185974145, 359010583, 576271053, $ \newline $781064029,  903961423, 900492886, 776270805, 580939911, 377892743, 213530461, 104552833, 44163497,  $ \newline $ 15985154, 4910774, 1263620, 267378, 45321, 5918, 559, 34, 1)$\\
       \hline
    \end{tabular}
    \caption{The first few primitive $f$-vectors of $CRY_n.$}
    \label{tab::primitive_f_vectors_of_CRY_n}
\end{table}

\begin{figure}
    \centering
    \includegraphics[width = \textwidth]{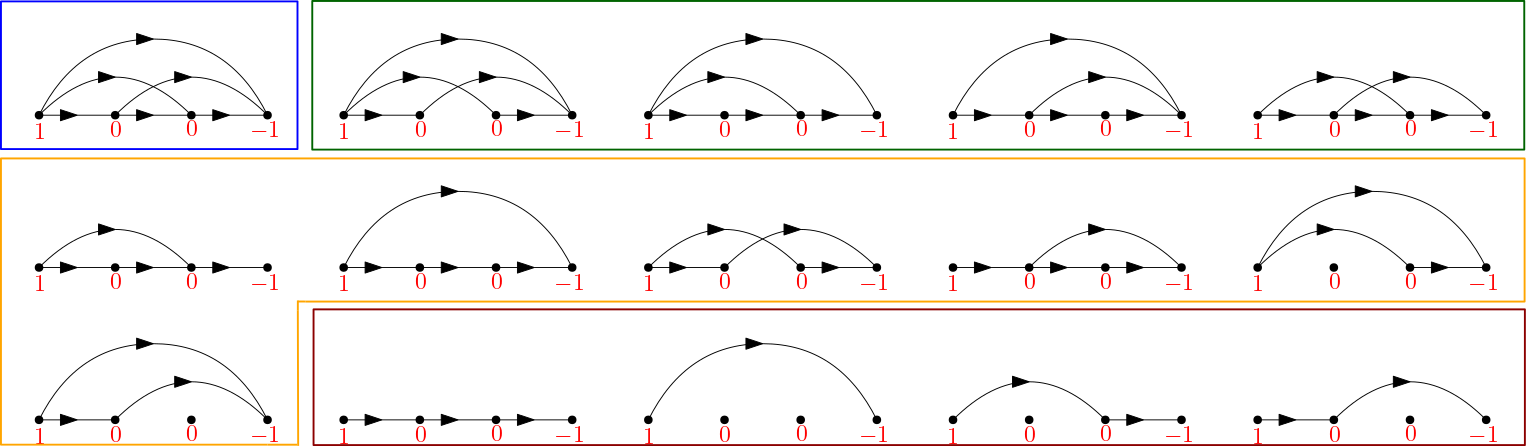}
    \caption{The elements of $\Omega_3$ grouped by first Betti number, corresponding to the $f$-vector $(1,4,6,4,1)$ of $CRY_3$ (but excluding the empty face which would correspond to the empty graph). The primitive $f$-vector $(0,1,4,4,1)$ corresponds to the number of graphs in each grouping which use all vertices.}
    \label{fig::faces_of_CRY_3}
\end{figure}

\subsection{Fishburn Matrices} A \defn{Fishburn matrix} is an upper-triangular matrix with natural number entries such that no column nor row is the zero vector. A \defn{primitive} Fishburn matrix is a Fishburn matrix using only entries in $\{0,1\}$. The redundant use of \textit{primitive} for these various objects is due to their connection to one another. Namely, Jel\'inek \cite{Jelinek_2011} observed that $\Omega_n'$ is in bijection with the set of primitive Fishburn matrices of size $n$ via the map that sends a graph to the upper-triangular $0$-$1$ matrix determined by the support of its edges (c.f. Figure~\ref{fig::Fishburn_matrix_example}). 

\begin{figure}
    \centering
    \includegraphics[width = 10cm]{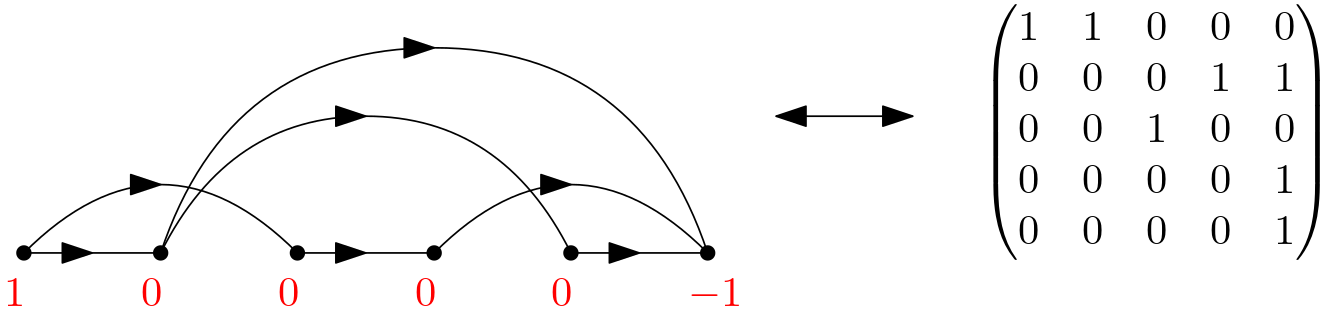}
    \caption{The correspondence between elements of $\Omega_n'$ and primitive Fishburn matrices.}
    \label{fig::Fishburn_matrix_example}
\end{figure}

Jel\'inek \cite{Jelinek_2011} further gave a multivariate generating function for the number of primitive Fishburn matrices keeping track of five different statistics:

\begin{equation}
\label{eq::Jelinek_generating_function}
    G(t,v,w,x,y) = \sum_{n \geq 0} t^{n+1} \frac{(x+1)(y+1)^n - 1}{1 + t((v+1)(w+1)^n - 1)} \prod_{i = 0}^{n-1} \frac{(v+1)(w+1)^i - 1}{1 + t((v+1)(w+1)^i - 1)}
\end{equation}
where $t$ keeps track of the number of rows in each matrix, $x$ keeps track of the corner cell, $y$ keeps track of the sum of the first row excluding the corner cell, $v$ keeps track of the sum of the last column excluding the corner cell, and $w$ keeps track of the sum of all entries excluding those in the first row and last column \cite{Jelinek_2011}.  

More general instances of Fishburn matrices do not appear to have any immediate relationship with more general faces of $CRY_n$.

\begin{remark}
    \textbf{Notations and conventions:}  Our vector ${\bf a}$ used in this paper is often denoted ${\widetilde{\bf a}}$ in the flow polytope literature, as it does not account for the last vertex whose netflow is predetermined by the first $n$ entries. Moreover, we note that in the case of $a_i \geq 0$ for all $i$ as we are assuming in this manuscript,  a consequence of Theorem~\ref{thm::Hille} is that the combinatorial equivalence class of $\mathcal{F}_G\left({\bf a}, -\sum_{i = 1}^n a_i\right)$ is completely determined by the support of ${\bf a}$. Hence we may assume for the rest of the paper that ${\bf a} \in \set{0,1}^n$. An excellent source for any other unexplained terms and notation is \cite{Benedetti_transactions_2018}.
\end{remark}

\section{$f$-polynomials of $\Flow_n({\bf a})$}
\label{sect::main_results}

We now describe various results that build towards Theorem~\ref{thm::main_result}. 

\subsection{Formulas as sums over subsets}

In \cite{andresen_kjeldsen_1976}, the authors define certain sequences of numbers which prove useful for the exact enumeration of the sets $\Omega_n$ and $\Omega_n'$ (here we require a change of convention to non-increasing sequences instead of non-decreasing sequences). 

\begin{Def}[\cite{andresen_kjeldsen_1976}]
    Let $S_{n,m}$ be the set of all sequences $(i_1, \ldots i_n) \in \NN^n$ having length $n$ such that:
    
    \begin{enumerate} 
    \item[(i)] $i_1 = n-m$, \hspace{1.5cm}
         \item[(ii)] $i_n = 1$, \hspace{1.5cm}
         \item[(iii)] $i_j \geq i_{j+1} \geq i_j - 1$ for all $j < n$.
    \end{enumerate}
\end{Def}

For our purposes, it will be simpler to think of the sequences in $S_{n,m}$ as subsets of $\nset{n} := \set{1, \ldots, n}$ through the following correspondence.

\begin{Prop}
\label{lem::sequences_to_subsets}
    The map $\desc: S_{n,m} \rightarrow \binom{\nset{n-1}}{m}$ mapping a sequence in $S_{n,m}$ to the set of indices of its descents is a bijection. 
\end{Prop}
\begin{proof}
    Let $s = (i_1, \ldots , i_n)$ be an element of $S_{n,m}$. Then in particular $s$ is a sequence such that $i_1 = n-m, i_n = 1$, and every element of the sequence is either equal to the previous element or exactly one less than it. In particular, $s$ is a non-increasing sequence with exactly $m$ descents $j_1, \ldots j_m$ (note that a \textit{descent} $j_k$ is an index such that $i_{j_k} > i_{j_k + 1}$). We define $f: S_{n,m} \rightarrow \binom{[n]}{m}$ via the rule $f(s) = \set{j_1, \ldots j_m}$.  To see that $f$ is a bijection, it suffices to note that every such sequence $s$ is completely determined by the locations of its descents, and every $m$-element subset of $[n]$ can occur as the descent set of some $s$. 
\end{proof}

From here on we will be more interested in the inverse function of Lemma~\ref{lem::sequences_to_subsets}, and hence will denote by $\seq_n: \nset{n} \rightarrow \bigsqcup_{m = 0}^n S_{n+1,m}$ the map that  takes a subset of $\nset{n}$ to its corresponding non-increasing sequence of length $n+1$. 

\begin{Ex}
    The following is an example of $\seq_4$ applied to subsets of the set $[4]$ of cardinality $2$:
    
    \begin{tabular}{ccc}
       $\seq_4(\set{1,2}) = (3,2,1,1,1),$  & $\seq_4(\set{1,3})  = (3,2,2,1,1),$ & $\seq_4(\set{1,4})  = (3,2,2,2,1),$ \\
        $\seq_4(\set{2,3})  = (3,3,2,1,1), $& $\seq_4(\set{2,4})  = (3,3,2,2,1), $& $\seq_4(\set{3,4})  = (3,3,3,2,1)$.
    \end{tabular}
    
\end{Ex}

\begin{remark}
\label{remark::how_to_compute_seq}
    Note that the $j$th element of the tuple $\seq_n(S)$ is exactly $1 + \#\{s \in S \, | \, s \geq j\}$. The parameter $n$ only determines the length of the tuple $\seq_n(S)$.
\end{remark}

These are all the ingredients we need to write down a first formula for $\widetilde{f}^{(n)}({\bf a};x)$.

\begin{Lem}
\label{lem::primitive_f_vec_arbitrary_net_flow_sum_over_sets}
    For all $n \in \NN$ and non-negative ${\bf a}$ of length $n$, a formula for $\widetilde{f}^{(n)}({\bf a};x)$ (that is, the primitive $f$-vector of  $\Flow_n({\bf a})$ written as a polynomial in $x$) is given by:
    \begin{equation}
    \label{eq::primitive_f_vec_sum_over_sets}
        \widetilde{f}^{(n)}({\bf a};x) = \frac{1}{x^n}\sum_{\supp({\bf a}') \subseteq S \subseteq \nset{n-1} }(-1)^{|S| + n + 1} \prod_{j \in \nset{n}}((x+ 1)^{\seq_{n-1}(S)_j} - 1 ) 
    \end{equation}
 where ${\bf a}' = (a_2, a_3, \ldots a_n)$, $\supp$ is the support function (namely $\supp({\bf a}')$ returns the set of indices $j$ such that $a_{j+1} \neq 0$).
\end{Lem}
\begin{proof}
    The idea of the argument is as follows. First, we will start with the set of all primitive subgraphs of $K_{n+1}$ (not just ${\bf a}$-valid ones). We will then apply the principle of inclusion and exclusion in order to obtain the set of primitive subgraphs that are also ${\bf a}$-valid.

    With this in mind, let $T$ be a subset of $\set{v_2, \ldots v_n}$ and associate to $T$ its indicator set $S_T \subseteq \nset{n-1}$ in the canonical way (namely $i \in S_T$ if and only if $v_{i+1} \in T$). For each such $S$, define $R_S$ to be the set of primitive subgraphs of $K_{n+1}$ such that:
    \begin{enumerate}
        \item $outdeg(v_{i}) > 0$ for all $i \in [n]$, where $outdeg$ is the out-degree of the vertex $v_i$, and \label{list::item_1}
        \item $i \in S^c$ implies $indeg(v_{i+1}) = 0$, where $indeg(v_{i+1})$ is the in-degree of vertex $v_{i+1}$.
    \end{enumerate}
    In other words, $R_S$ is the set of primitive subgraphs of $K_{n+1}$ for which all vertices have nonzero out-degree and for which $S$ is the set of indices of vertices which are allowed to have non-zero in-degree. We claim that $S_1 \subseteq S_2$ implies $R_{S_1} \subseteq R_{S_2}$. 

    To prove the claim, let $S_1 \subseteq S_2$ and let $H$ be a graph in $R_{S_1}$. To show that $H$ is also in $R_{S_2}$, we need to prove that for every $i \in S_2^c$, the in-degree of $v_{i+1}$ in $H$ is $0$. So let us take $i$ in $S_2^c$. Now since $S_1 \subseteq S_2$, it follows that $S_2^c \subseteq S_1^c$ and so $i$ is an element of $S_1^c$. However, since $H$ is in $S_1$ by assumption, this in turn means that the $indeg(v_{i+1}) = 0$. Hence we have shown that $i \in S_2^c$ implies $indeg(v_{i+1}) = 0$, and so $H$ in $R_{S_2}$. 
    
    Now by the claim, the cardinality of the set $\Prim_{\bf a}$ of ${\bf a}$-valid primitive subgraphs of $K_{n+1}$ may be found via inclusion-exclusion as follows: 
    \begin{equation}
    \label{eq::lemma_proof_eq1}
        |\Prim_{\bf a}| = \sum_{\supp({\bf a}') \subseteq S \subseteq \nset{n-1}  }(-1)^{|S| + n + 1} |R_S|,
    \end{equation}
    where the lowest set in the interval of summation is $\supp({\bf a}')$ since the elements of any subset of $R_{\supp({\bf a}')}$ are ${\bf a}$-valid. Next, if we let $r_S(x)$ be the generating function over the set $R_S$ that keeps track of the sum of all out-degrees of each graph in $R_S$, then we claim that:
    \begin{equation}
    \label{eq::lemma_proof_eq2}
        r_S(x) := \sum_{H \in R_S} x^{\sum_{v_i \in V(H)} outd(v_i)} = \prod_{j \in \nset{n}}((x+ 1)^{\seq_{n-1}(S)_j} - 1 ).
    \end{equation}

    To prove this claim, let us fix some $S$ such that $\supp({\bf a}') \subseteq S \subseteq \nset{n-1} $ and compute $r_S(x)$ explicitly. For fixed $j \in [n]$, consider the $j$th vertex from last, namely vertex $v_{n+1 - j}$. Among the vertices $v_{n+2 - j}, \ldots, v_{n+1}$ that come after $v_{n+1 - j}$, those that $v_{n+1 - j}$ may be adjacent to are exactly those whose indices are determined by $S$. Hence vertex $v_{n+1 - j}$ contributes a factor of $(x+1)^k - 1$ where $k$ is the number $1 + \#\{s \in S \, | \, s \geq j\}$ (where we use inclusive inequality since the numbers in $S$ are the indices of the vertices shifted by $1$). Note that we subtract by $1$ here since the term $1$ corresponds to the choice that $v_{n+1 - j}$ has out-degree $0$ which violates condition~(\ref{list::item_1}) of the definition of $R_S$. However, by Remark~\ref{remark::how_to_compute_seq}, this number $k$ is exactly the $(n+1 - j)$th index of $\seq_{n-1}(S)$ (where the subscript is $n-1$ since vertex $v_{n+1}$ must have out-degree $0$ and vertex $v_{n}$ must have out-degree $1$). Finally, by condition~(\ref{list::item_1}) of the definition of $R_S$ above, every vertex has nonzero out-degree, meaning the contribution of one vertex to $r_S(x)$ will be independent of all the others. After re-indexing, we find that $r_S(x)$ is exactly as written in Equation~\eqref{eq::lemma_proof_eq2}. 

    Combining equations \eqref{eq::lemma_proof_eq1} and \eqref{eq::lemma_proof_eq2} gives a generating function over the set $\Prim_{\bf a}$ keeping track of the sum of all out-degrees of each graph. Finally, since our graphs are primitive, the first Betti number of each graph is exactly the sum of all out-degrees minus $n$, from which the final formula follows. 
\end{proof}

\begin{figure}
    \centering
    \includegraphics[width = \textwidth]{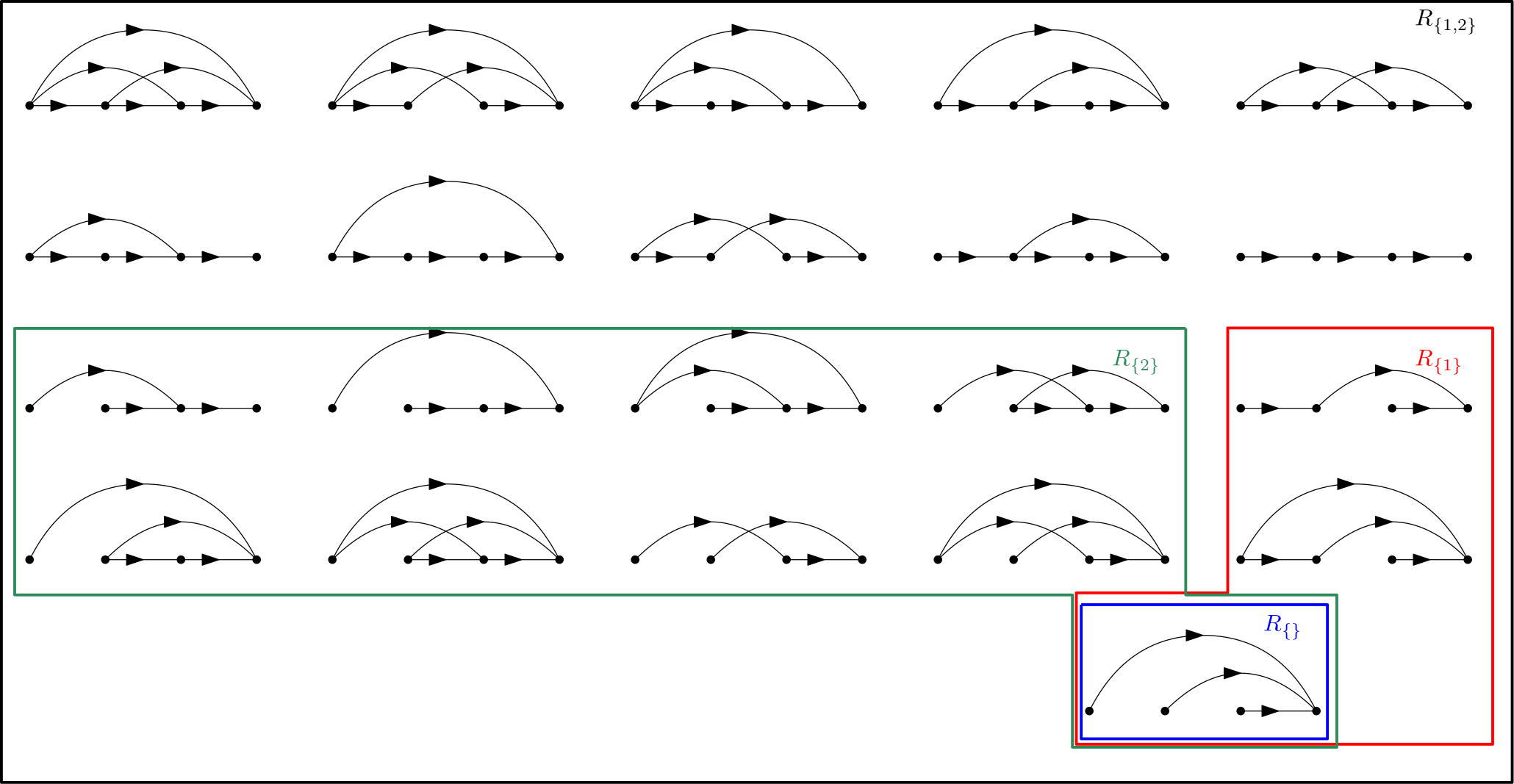}
    \caption{The sets of primitive graphs $R_S$ for $n = 3$, as described in the proof of Lemma~\ref{lem::primitive_f_vec_arbitrary_net_flow_sum_over_sets}. The nested boxes illustrate the claim that $S_1 \subseteq S_2$ implies $R_{S_{1}} \subseteq R_{{
    S_2}}$ in the proof of the lemma.}
    \label{fig::nested_sets_of_graphs}
\end{figure}

\begin{Ex}
Fix $n = 3$ and consider ${\bf a} = (1,0,0)$, i.e. the case of $CRY_3$. By the lemma we have:
\begin{multline}
\label{eq::example_calculation_f_poly_cry_3}
    \widetilde{f}^{(3)}((1,0,0); x) =  ((x+1)^3 - 1)\cdot((x+1)^2 - 1)\cdot((x+1)^1 - 1)
    -((x+1)^2 - 1)\cdot((x+1)^1 - 1)\cdot((x+1)^1 - 1)\\
    - ((x+1)^2 - 1)\cdot((x+1)^2 - 1)\cdot((x+1)^1 - 1)
    +((x+1)^1 - 1)\cdot((x+1)^1 - 1)\cdot((x+1)^1 - 1)
\end{multline}
which simplifies to:
\begin{align*}
    \widetilde{f}^{(3)}((1,0,0); x) &=  x^6 + 4x^5 + 4x^4 + x^3
\end{align*}
and agrees with Table~\ref{tab::primitive_f_vectors_of_CRY_n} after shifting exponents, as expected. Each of the terms in Equation~\eqref{eq::example_calculation_f_poly_cry_3} above corresponds to one of the groupings of graphs in Figure~\ref{fig::nested_sets_of_graphs}. For example, consider the set $R_{\{2\}}$ of primitive graphs in $K_4$ depicted in the green grouping of Figure~\ref{fig::nested_sets_of_graphs}. We compute $\seq_{2}(\{ 2\}) = (2,2,1)$, since there is one descent in index $2$ and the tuple has length $3$. Hence by the proof of Lemma~\ref{lem::primitive_f_vec_arbitrary_net_flow_sum_over_sets} we have:
$$r_{\{2\}}(x) = ((x+1)^2 - 1)\cdot((x+1)^2 - 1)\cdot((x+1)^1 - 1).$$
After expanding this becomes:
$$r_{\{2\}}(x)  = x^5 + 4x^4 + 4x^3.$$
Subtracting each exponent by $n = 3$, this computation says that there should be $1$ graph having first Betti number $2$, $4$ graphs having first Betti number $1$, and $4$ graphs having first Betti number $0$, exactly as seen in the green grouping $R_{\{2\}}$ of Figure~\ref{fig::nested_sets_of_graphs}. The other terms are computed similarly.
\end{Ex}

The next result describes how the $f$-polynomial of $\Flow_n({\bf a})$ may be obtained easily as a sum of primitive $f$-polynomials.

\begin{Lem}
    \label{lem::generalized_psi_xi_relationship}
    For all $n \in \NN$ and non-negative ${\bf a}$ of length $n$:
\begin{equation}
    f^{(n)}({\bf a};x) = \frac{1}{x} + \sum_{{\bf b} \leqc {\bf a}} k_{{\bf a}, {\bf b}} \widetilde{f}^{(\ell(\bf b))}({\bf b};x)
\end{equation}
where ${\bf b} \leqc {\bf a}$ if ${\bf b}$ can be obtained from ${\bf a}$ by deleting some subset (possibly empty) of the zeros in ${\bf a}$, $\ell(\bf b)$ is the length of ${\bf b}$ and where $k_{{\bf a}, {\bf b}}$ is the number of ways of deleting $0$'s from ${\bf a}$ to obtain ${\bf b}$. 
\end{Lem}

\begin{proof}
    Let $F$ be a face of $\Flow_n({\bf a})$. If $F$ is the empty face, then it does not correspond to a primitive graph and hence contributes a term of  $\dfrac{1}{x}$ to $f^{(n)}({\bf a};x)$. Otherwise, $F$ is non-empty and hence corresponds to an ${\bf a}$-valid subgraph $H \subseteq K_{n+1}$ by Theorem~\ref{thm::Hille}. Let $S_H \subseteq \set{v_1, \ldots, v_{n+1}}$ be the set of vertices which are part of the support of a flow determining $H$. Then $H$ is a primitive graph when restricted to the vertex set $S_H$, hence is counted by $\widetilde{f}^{(|b|)}({\bf b};x)$ for some ${\bf b}$ determined by $S_H$. 
    
    The possible ${\bf b}$'s that can appear are exactly those described in the lemma statement. To see this, note that deleting the entries which are $0$ from the string is the same as restricting the graph $H$ to the vertices in the set $S_H$. On the other hand, entries of $1$ may not be deleted since every such entry corresponds to a source of flow and hence will be included in the support of any ${\bf a}$-valid graph. By definition of $k_{{\bf a}, {\bf b}}$, each string ${\bf b}$ appears $k_{{\bf a}, {\bf b}}$-many times.
\end{proof}

\begin{remark}
    We can give an explicit formula for $k_{{\bf a}, {\bf b}}$ after using the correspondence between $0$-$1$ vectors and compositions as described in Remark~\ref{rem::composition_posets_as_bitstrings}; see Lemma~\ref{lem::f_vec_as_sums_of_downsets_of_poset}.
\end{remark}

\begin{Ex}
    As an example, Lemma~\ref{lem::generalized_psi_xi_relationship} would give us the following relation:
    \begin{multline*}
        f^{(6)}(1,0,0,1,1,0;x) = \frac{1}{x} +\widetilde{f}^{(6)}(1,0,0,1,1,0;x) + 2\widetilde{f}^{(5)}(1,0,1,1,0;x) + \widetilde{f}^{(5)}(1,0,0,1,1;x)   \\ + \widetilde{f}^{(4)}(1,1,1,0;x)  + 2\widetilde{f}^{(4)}(1,0,1,1;x) + \widetilde{f}^{(3)}(1,1,1;x)
    \end{multline*}
    where the coefficient $2$ arises in front of  $\widetilde{f}^{(5)}(1,0,1,1,0;x)$, for example, as there are two ways to delete zeros that result in this input vector.
    
\end{Ex}

\subsection{Formulas as evaluations of sums of polynomials}
\label{subsect::formulas_as_evaluations_of_sums_of_quasisymmetric_polynomials}

We can rewrite Lemma~\ref{lem::primitive_f_vec_arbitrary_net_flow_sum_over_sets} as an evaluation of a certain quasisymmetric-like polynomial by using the standard bijection of subsets of $\nset{n-1}$ with integer compositions of $n$. Indeed, given a composition $\alpha$ and corresponding set $S_{\alpha}$ we define the multivariate polynomial:
\begin{equation}
\label{eq::polynomial_P}
    P_{\alpha}(x_1, \ldots, x_n) := \sum_{\beta \geqc \alpha} (-1)^{n - \ell(\beta)} {\bf x}^{\beta}  
\end{equation}
where ${\bf x}^{\beta} := x_1^{\beta_1}\cdots x_{\ell(\beta)}^{\beta_{\ell(\beta)}}$, and where the relation $\geqc$ is \textit{reverse refinement} on compositions (see Section~\ref{subsect::posets_of_compositions}).  

\begin{remark}
    The polynomial $P_{\alpha}$ may look familiar to the reader. Indeed, we recall that the \defn{monomial quasisymmetric functions}, $M_{\alpha}$, and Gessel's \defn{fundamental quasisymmetric functions}, $F_{\alpha}$, are defined in infinitely many variables $x_i$ respectively via:
\begin{equation*}
    M_{\alpha} := \sum_{i_1 < i_2 < \ldots < i_k} x_{i_1}^{\alpha_1} x_{i_2}^{\alpha_2} \cdots x_{i_k}^{\alpha_k}, \hspace{2cm}  F_{\alpha} := \sum_{\substack{i_1 \leq i_2 \leq \ldots \leq i_k \\ i_j < i_{j+1} \text{ if } j \in S_{\alpha}}} x_{i_1}x_{i_2} \cdots x_{i_k}.
\end{equation*}
A standard result of quasisymmetric functions describes how to write the monomial quasisymmetric functions in terms of Gessel's fundamental quasisymmetric functions and vice versa. Namely we have the equations (c.f. \cite[Sect 7.19]{EC2}\cite{Petersen_2005}):
\begin{equation}
    F_{\alpha} = \sum_{\beta \succeq \alpha} M_{\beta}, \hspace{2cm}  M_{\alpha} = \sum_{\beta \succeq \alpha} (-1)^{\ell(\beta) - \ell(\alpha)} F_{\beta}.
\end{equation}

Hence, the polynomial $P_{\alpha}(x_1, \ldots x_n)$ from above is exactly the expansion of $M_{\alpha}$ into the fundamental basis, except that we only keep the first term of each $F_{\beta}$; that is:

\begin{equation}
    P_{\alpha}(x_1, \ldots , x_n) = (-1)^{n + \ell(\alpha)}\sum_{\beta \succeq \alpha} (-1)^{\ell(\beta) - \ell(\alpha)} F_{\beta}(x_1, \ldots, x_{\ell(\beta)}) .
\end{equation}
\end{remark} 

The polynomials $P_\alpha$ capture all of the data needed to compute the primitive $f$-vector $\widetilde{f}_n({\bf a};x)$.

\begin{Lem}
   \label{lem::primitive_f_vec_as_quasi_sym}
    For all $n \in \NN$ and non-negative ${\bf a}$ of length $n$, let $\alpha$ be the composition of $n$ given by $\alpha = \revcomp({\bf a})$. Then the primitive $f$-vector of $\Flow_n({\bf a})$ written as a polynomial is given by:
    \begin{equation}
    \label{eq::primitive_f_vec_as_quasi_sym}
        \widetilde{f}^{(n)}({\bf a}; x) = \frac{1}{x^n}P_{\alpha}(x, (x+1)^2 - 1, \ldots, (x+1)^{n} - 1)
    \end{equation}
\end{Lem}

\begin{proof}
    The proof follows from Lemma~\ref{lem::primitive_f_vec_arbitrary_net_flow_sum_over_sets} by applying $\comp \circ \rev$ to every subset appearing in the sum Equation~\eqref{eq::primitive_f_vec_sum_over_sets}, where $\rev$ is as defined in Definition~\ref{def::revcomp} and $\comp$ is the standard map sending a subset of $[n]$ to a composition of $n+1$. 
    
    To begin, observe that the interval of summation $\supp({\bf a}') \subseteq S \subseteq \nset{n-1} $ from Equation~\eqref{eq::primitive_f_vec_sum_over_sets} translates under $\comp \circ \rev$ to the interval $\{ \beta \models n \, | \, \revcomp(\bf{a}) \leqa \beta\}$ of the poset of compositions $(\mathcal{C}_n, \leqa)$. To see this, note from Remark~\ref{rem::composition_posets_as_bitstrings} that subset inclusion translates under this bijection to reverse refinement of compositions; the largest set $[n-1]$ translates under the bijection to the most-refined composition, $(1^n)$; and the length of the composition $\alpha_S := (\comp \circ \rev)(S)$ is $\ell(\alpha_S ) = |S| - 1$.    

    All that remains is to show what happens to the factor $\prod_{j \in \nset{n}}((x+ 1)^{\seq_{n-1}(S)_j} - 1 )$ under $\comp \circ \rev$. For a fixed set $S$ and corresponding composition $\alpha_S := (\comp \circ \rev)(S)$. We claim that the $i$th entry of $\alpha_S$ is exactly the number of entries of $\seq(S)$ that are equal to $i$. To see this, note that if $S = \{s_1, \ldots, s_k \}$ where $s_1 < \ldots < s_k$, then $\alpha_S$  is exactly the composition $(n + 1 - s_k , s_k - s_{k-1}, \ldots , s_2 - s_1, s_1)$. Meanwhile,  we know from Remark~\ref{remark::how_to_compute_seq} that the $j$th element of the tuple $\seq_n(S)$ is exactly $1 + \#\{s \in S \, | \, s \geq j\}$. Hence the numbers $s_{i+1} - s_i$ which form the entries of $\alpha_S$ are giving exactly the lengths of the runs in $\seq(S)$ where the values are constant. Hence the product $\prod_{j \in \nset{n}}((x+ 1)^{\seq_{n-1}(S)_j} - 1 )$ corresponds under the map $\comp \circ \rev$ with the monomial ${\bf x}^{\alpha_S}|_{x_{i} = (x+1)^i - 1}$. This concludes the proof.
\end{proof}

\begin{Ex}
    Here we demonstrate an example of the correspondence between terms of Equation~\eqref{eq::primitive_f_vec_sum_over_sets} and terms of Equation~\eqref{eq::primitive_f_vec_as_quasi_sym} as described in the proof of Lemma~\ref{lem::primitive_f_vec_as_quasi_sym}.  Consider $n = 7$ and the netflow vector ${\bf a} = (1,0,0,1,1,0,1)$. Computing the contribution of the first term in Equation~\eqref{eq::primitive_f_vec_sum_over_sets}, we find that $\supp({\bf a}) = \{3,4,6\}$. Applying $\seq$ to this set, we compute $\seq_7(\{3,4,6\}) = (4,4,4,3,2,2,1)$. Hence the first term of Equation~\eqref{eq::primitive_f_vec_sum_over_sets} is: 
    $$-((x+1)^4 - 1)^3((x+1)^3 - 1)^1((x+1)^2 - 1)^2((x+1)^1 - 1)^1.$$ 
    Note that the sign is negative since $|\{3,4,6\}| + n  + 1 = 11$ is odd. 
    
    On the other hand, to compute the first term of Equation~\eqref{eq::primitive_f_vec_as_quasi_sym} we compute $\revcomp({\bf a}) = (1,2,1,3)$. This produces the monomial $x_1^1x_2^2x_3^1x_4^3$, which we then evaluate at $x_i = (x+1)^i - 1$. The result is 
    $$-((x+1)^1 - 1)^1((x+1)^2 - 1)^2((x+1)^3 - 1)^1((x+1)^4 - 1)^3,$$ where the sign is negative since $n - \ell((1,2,1,3))$ = 3 is odd.
\end{Ex}

We may now combine Lemma~\ref{lem::generalized_psi_xi_relationship} and Lemma~\ref{lem::primitive_f_vec_as_quasi_sym} to obtain an explicit formula for the $f$-vector of  $\Flow_n({\bf a})$ for ${\bf a}$ non-negative.

\begin{Lem}
\label{lem::f_vec_as_sums_of_downsets_of_poset}
     For all $n \in \NN$ and non-negative ${\bf a} \in \NN^n$, let $\alpha$ be the composition of $n$ given by $\alpha = \revcomp({\bf a})$. Then the $f$-vector of $\Flow_n({\bf a})$ written as a Laurent polynomial is given by:
    \begin{equation}
        f^{(n)}({\bf a}; x) = \frac{1}{x} + \frac{1}{x^n}\sum_{\beta \leqc \alpha} x^{|\alpha| - |\beta|}\left(\prod_{i = 1}^{\ell(\alpha)}\binom{\alpha_i - 1}{\alpha_i - \beta_i}\right)  P_{\beta}(x, (x+1)^2 - 1, \ldots, (x+1)^{|\beta|}  - 1)
    \end{equation}
\end{Lem}

\begin{proof}
    Combining the results of Lemma~\ref{lem::generalized_psi_xi_relationship} and Lemma~\ref{lem::primitive_f_vec_as_quasi_sym} we obtain:
    \begin{equation*}
        f^{(n)}({\bf a};x) = \frac{1}{x} + \sum_{{\bf b} \leqc {\bf a}} \frac{1}{x^{|{\bf b}|}}k_{{\bf a}, {\bf b}} P_{\beta}(x, (x+1)^2 - 1, \ldots, (x+1)^{|\beta| - 1})
    \end{equation*}
    where $\leqc$ is the partial order on $0$-$1$ vectors described in Lemma~\ref{lem::generalized_psi_xi_relationship}  and $\beta = \revcomp(\bf b)$. Translating both ${\bf a}$ and ${\bf b}$ into compositions via $\revcomp$, we find that the resulting partial order is exactly that of $(\mathcal{C}, \leqc)$ (see Figure 2 and Remark~\ref{rem::composition_posets_as_bitstrings}). Indeed, the number of parts of $\revcomp({\bf a})$ corresponds to the number of $1$'s appearing in ${\bf a}$, and deleting a $0$ in ${\bf a}$ corresponds to decreasing the corresponding part of $\revcomp{({\bf a})}$ by $1$. In Lemma~\ref{lem::generalized_psi_xi_relationship} we are only able to delete $0$'s and not $1$'s, hence the number of parts of $\revcomp({\bf b})$ must be the same as the number of parts of $\revcomp({\bf a})$. Next, we note that the product of binomial coefficients keeps track of the number of ways of deleting $0$'s from ${\bf a}$ that result in the same ${\bf b}$, hence $k_{{\bf a}, {\bf b}} = \prod_{i = 1}^{\ell(\alpha)}\binom{\alpha_i - 1}{\alpha_i - \beta_i}$. Finally, the factor of $x^{|\alpha| - |\beta|}$ arises as a result of taking the common denominator of all terms.
\end{proof}

All of the work has now been done in order to prove our main result.

\MainResult

\begin{proof}
    Lemma~\ref{lem::f_vec_as_sums_of_downsets_of_poset} gives $f^{(n)}({\bf a}; x)$ as an evaluation of a linear combination of $P_{\beta}$'s coming from downsets of the poset $(\mathcal{C}, \leqc)$. However, each $P_{\beta}$ is also a sum over compositions (Equation~\eqref{eq::polynomial_P}). We may expand the $P_{\beta}$'s to obtain:
    \begin{align}
        f^{(n)}({\bf a}; x) &= \frac{1}{x} + \frac{1}{x^n}\sum_{\beta \leqc \alpha} x^{|\alpha| - |\beta|}\left(\prod_{i = 1}^{\ell(\alpha)}\binom{\alpha_i - 1}{\alpha_i - \beta_i}\right)  \sum_{\gamma \geqa \beta} (-1)^{|\beta| - \ell(\gamma)} {\bf x}^{\gamma}|_{x_i = (x+1)^i - 1} \nonumber\\
        &= \frac{1}{x} + \frac{1}{x^n}\sum_{\beta \leqc \alpha} \sum_{\gamma \geqa \beta} \left(\prod_{i = 1}^{\ell(\alpha)}\binom{\alpha_i - 1}{\alpha_i - \beta_i}\right)(-1)^{|\beta| - \ell(\gamma)} x_1^{|\alpha| - |\beta|} {\bf x}^{\gamma}|_{x_i = (x+1)^i - 1} \label{eq::step_of_main_theorem}
    \end{align}
Now note that $\leqc$ preserves the length of compositions and $\leqa$ preserves the size of compositions. From this we can conclude that each of the monomials $x_1^{|\alpha| - |\beta|} {\bf x}^{\gamma}|_{x_i = (x+1)^i - 1}$ has total degree $|\alpha| - |\beta| + |\gamma| = |\alpha| = n$, and so the sum (before taking into account the denominator outside) is a homogeneous polynomial of degree $n$. In particular, this means that every term of Equation~\eqref{eq::step_of_main_theorem} is indexed by a composition $\gamma' \composition n$. 

Our strategy to prove the result will be to prove a stronger statement instead: namely that the two multivariate polynomials indexed by compositions $\gamma'$ are equal before evaluating. In particular, we make this claim a separate lemma (Lemma~\ref{lem::helper_lemma_to_main_result}). The result will follow from Lemma~\ref{lem::helper_lemma_to_main_result} by evaluating Equation~\eqref{eq::main_result_proof_multivar_identity} at $x_i = (x+1)^i - 1$ for all $i$.
\end{proof}

\begin{Lem}
\label{lem::helper_lemma_to_main_result}
    \begin{equation}
\label{eq::main_result_proof_multivar_identity}
    \sum_{\beta \leqc \alpha} x_1^{|\alpha| - |\beta|}\left(\prod_{i = 1}^{\ell(\alpha)}\binom{\alpha_i - 1}{\alpha_i - \beta_i}\right)  P_{\beta}(x_1, x_2, \ldots, x_{|\beta|}) = \sum_{\alpha \leqa \beta}  (-1)^{\ell(\alpha) - \ell(\beta)} \prod_{i = 1}^{\ell(\beta)}x_i(x_i - x_1)^{\beta_i - 1}, 
\end{equation}
\end{Lem}

The proof of Lemma~\ref{lem::helper_lemma_to_main_result} is elementary (algebraic) yet technical; to maintain readability and flow we include its proof only later in Appendix~\ref{sect::appendixA}.

\section{$f$-polynomials and generating functions for $CRY_n$}

\label{sect::formulas_for_CRY_n}

Given the significance of $CRY_n$ in the research community, we dedicate this section to the explicit formulas for the $f$-polynomials of $CRY_n$ obtained by specializing the results in the previous section. We first obtain the following result by setting ${\bf a} = (1,0, \ldots, 0)$ in Lemma~\ref{lem::primitive_f_vec_as_quasi_sym}. We remark that it is a generalization of \cite[Prop. 3.2]{andresen_kjeldsen_1976} which one can recover by setting $x = 1$.

\begin{Cor}
    Let $\widetilde{f}^{(n)}(x)$ be the primitive $f$-polynomial of $CRY_n$. Then for all $n \geq 1$:
    \begin{equation}
        \widetilde{f}^{(n)}(x) = \frac{1}{x^n}\sum_{m = 0}^{n-1}(-1)^m\pi_{n-m}(x)\cdot h_m(x, (x+1)^2 - 1, \ldots , (x+1)^{n-m} -1).
    \end{equation} 
\end{Cor}

\begin{proof}
    In the case of $CRY_n$, ${\bf a} = (1, 0, \ldots, 0)$, hence $\revcomp({\bf a})$ is the composition $(n)$. Hence $P_{\alpha}(x_1, \ldots, x_n)$ in Lemma~\ref{lem::primitive_f_vec_as_quasi_sym} has a term for every integer composition of $n$. Factoring out $x_1 \cdots x_{n-m}$ from the terms coming from level $n-m$ in the poset of compositions by reverse refinement leaves $h_{m}(x_1, \ldots, x_{n-m})$. We then evaluate each $x_i$ in the same way as Lemma~\ref{lem::primitive_f_vec_as_quasi_sym}.
\end{proof}

We omit the proof of the following corollary as it follows similarly, except by specializing to ${\bf a} = (1,0, \ldots ,0)$ in Theorem~\ref{thm::main_result} instead of Lemma~\ref{lem::primitive_f_vec_as_quasi_sym}.

\CorSpecializeCRY

The following result is a specialization of Lemma~\ref{lem::generalized_psi_xi_relationship} to the case of ${\bf a} = (1, 0, \ldots, 0)$ and further gives \cite[eq. (1)]{andresen_kjeldsen_1976} by summing over all $d$:
\begin{Cor}
\label{cor::CRY_binomial_f_vec_f_primitive_vec_relationship} Entries of the $f$-vector and primitive $f$-vector of $CRY_n$ satisfy: $f^{(n)}_d  = \sum_{i = 0}^{n-1}\binom{n-1}{i}\widetilde{f}^{(n - i)}_d.$
\end{Cor}
\begin{proof}
    By Lemma~\ref{lem::generalized_psi_xi_relationship}, we can obtain $f^{(n)}({\bf a}; x)$ from $\widetilde{f}^{(n)}({\bf a}; x)$ by summing over all subsets of zeros in ${\bf a}$. For $CRY_n$, ${\bf a} = (1, 0, \ldots, 0)$, so all possible subsets of $0$'s occur. 
\end{proof}

Next, we present an intriguing relationship between $f^{(n)}$ and $\widetilde{f}^{(n)}$ described in the following Theorem.

\fVectorPrimitiveRelationship

\begin{proof}
        The proof is analogous to that of \cite[Prop. 4.1]{andresen_kjeldsen_1976}.  We have:
\begin{align*}
    \begin{split}
         \frac{(1+x)^{n}}{x}\widetilde{f}^{(n-1)} - f^{(n}(x) =  \frac{(1+x)^{n}}{x}\cdot \frac{1}{x^{n-1}} \sum_{m = 0}^{n-2} (-1)^m \pi_{n-m-1}(x)h_m((x+1)^1 - 1, \ldots , (x+1)^{n-m-1} - 1) ) \\ - \frac{1}{x} - \frac{1}{x^n} \sum_{m = 0}^{n-2} (-1)^m (1 + x)^m \pi_{n-m}(x)h_m((x+1)^1 - 1, \ldots, (x+1)^{n-m-1} - 1).
    \end{split}
\end{align*}
Which after algebraic manipulations simplifies to:
\begin{equation*}
     \frac{(1+x)^{n}}{x}\widetilde{f}^{(n-1)} - f^{(n}(x) = - \frac{1}{x} + \frac{1}{x^n}\sum_{m = 0}^{n-2} \pi_{n-m-1}(x) h_m(-(x+1)^2 + (x+1), \ldots, -(x+1)^{n-m} + (x+1))
\end{equation*}
 We may use the path model for the complete homogeneous symmetric functions to rewrite this expression as:
    \begin{equation}
    \label {eq::proof_of_lem_4.8}
        \frac{(1+x)^{n}}{x}\widetilde{f}^{(n-1)} - f^{(n}(x) = -\frac{1}{x} + \frac{1}{x^n}\sum_{i = 0}^{\infty} (N(x))^{n-1}_{1,i}
    \end{equation}
    where $N(x)$ is the weighted adjacency matrix for the infinite path graph having a self loop at each vertex, with loop at vertex $i$  ($i \geq 2$) having weight $(x+1)^i - 1$ and with edge $(i, i+1)$ having weight $-(x + 1)^{i+1} + (x+1)$. In other words, $N(x)$ has the following form:
        \begin{equation*}
        N(x) := \setlength\arraycolsep{2pt}\begin{bmatrix}
            0 & x & 0  & \cdots & 0 \\
            0 & -(x+1)^2 + (x+1) & (x+1)^2 - 1  & \cdots & 0 & \cdots\\
              \vdots & \vdots & \vdots  & \ddots & 0 & \cdots\\
               0 & 0 & 0  & -(x+1)^i + (x+1) & (x+1)^i - 1 & \cdots \\
                \vdots & \vdots & \vdots  & \vdots & \ddots & \cdots\\
        \end{bmatrix}.
    \end{equation*}
    A simple induction shows that $\sum_{i = 0}^{\infty} (N(x))^{n-1}_{1,i} = x^{n-1}$, and after plugging into Equation~\eqref{eq::proof_of_lem_4.8}, gives the result.
\end{proof}

We conclude this section with a result that packages all of the face numbers of $CRY_n$ for varying $n$ and varying dimension $d$ into a single generating function.

\CorGeneratingFunction

\begin{proof}
The result may be established either by using Corollary~\ref{cor::CRY_f_vec_formula} and the generating function for the complete homogeneous symmetric functions (c.f. \cite{EC2}), or from the multivariate generating function of Fishburn matrices due to Jel\'inek, Equation~\eqref{eq::Jelinek_generating_function} (\cite[Thm. 2.1]{Jelinek_2011}), together with Theorem~\ref{thm::f_vector_product_result}. Preferring its simplicity, we present the latter. 

    Using the relationship established between primitive  subgraphs of $K_{n+1}$ and primitive Fishburn matrices (see Figure~\ref{fig::Fishburn_matrix_example}), the generating function for primitive $f$-vectors of $CRY_n$ can be seen to be
    \begin{equation}
        \sum_{n = 0}^{\infty} \widetilde{f}_n(x) t^n = G(t,x,x,x,x),
    \end{equation}
    where $G$ is the generating function of non-empty primitive interval orders as defined in \cite{Jelinek_2011}. Using Equation~\eqref{eq::Jelinek_generating_function} (Theorem 2.1 of \cite{Jelinek_2011}) we compute:
    $$G(t,x,x,x,x) = \sum_{n \geq 0} t^{n+1} \prod_{i = 0}^{n}\frac{(x+1)^{i+1} - 1}{1 + t[(x+1)^{i+1} - 1]}.$$
    Finally, by Theorem~\ref{thm::f_vector_product_result} it follows that:
    \begin{align*}
        \sum_{i = 0}^{\infty}f_n(x)t^n &= \frac{1}{x}G((1+x)t, x,x,x,x) \\
        &= \frac{1}{x}\sum_{n = 0 }^{\infty}t^{n+1}(1+x)^{n+1}\prod_{i=0}^{n}\frac{(x+1)^{i+1} - 1}{1 + t[(x+1)^{i+2} - (x+1)]}\\
        &= F(t,x)
    \end{align*}
    as claimed.
\end{proof}

 \section{Direct enumerative results and recurrences}

 \label{sect::direct_enumerative_results_and_recurrences}

In this section, we return to the original problem of enumerating faces of $\Flow_n({\bf a})$ and see how far direct counting techniques take us. In particular, for a few special cases we find formulas for face numbers $f_d^{(n)}({\bf a})$ that are more computationally efficient than computing the full $f$-vector as given by Theorem~\ref{thm::main_result}.

To begin, it was known to Chan--Robbins--Yuen \cite{Chan_Robbins_Yuen_2000} that their eponymous polytope has $2^{n-1}$ vertices. We provide a brief alternate proof of this result here to demonstrate how Theorem~\ref{thm::Hille} can be used to enumerate faces of flow polytopes. The rest of the results in this section arise similarly.

 \begin{Prop}[\cite{Chan_Robbins_Yuen_2000}]
 Let $f^{(n)} = (f^{(n)}_{d})_{d \geq -1}$ be the $f$-vector of $CRY_n$. Then $f^{(n)}_0 = 2^{n-1}.$
 \end{Prop}

 \begin{proof}
 Vertices of $CRY_n$ are $0$-dimensional faces, hence by Theorem~\ref{thm::Hille} we need to find $(1,0,\ldots, 0, -1)$-valid graphs of $K_{n+1}$ that have first Betti number equal to $0$. These are exactly paths from $v_1$ to $v_{n+1}$, and each such path is completely determined by the support of edges it uses from the set $\set{(v_1, v_2), (v_2, v_3), \ldots, (v_{n}, v_{n+1})}$. Hence the vertices of $CRY_n$ are in bijection with words of length $n-1$ in the alphabet $\set{0,1}$, which gives the result.
 \end{proof}

 The next result is proved similarly, except that instead of counting paths we count graphs having first Betti number $1$ and a directed path from $v_1$ to $v_{n+1}$. 

 \begin{Prop}
 \label{prop::Chan_Robbins_Yuen_original_number_vertices}
 Let $f^{(n)} = (f^{(n)}_{d})_{d \geq -1}$ be the $f$-vector of $CRY_n$. Then the number $f^{(n)}_1$ of edges have closed form expression:
 \begin{equation}
 \label{eq::closed_form_dim_1_faces}
     f^{(n)}_1 = 2\cdot 3^{n-1} - (n+3)\cdot 2^{n-2}.
 \end{equation}
 \end{Prop}

 \begin{proof}
    First we will prove that the numbers   $f^{(n)}_1$ satisfy the recurrence relation:
 \begin{equation}
     f^{(n)}_1 = f^{(n-1)}_1 + f^{(n-2)}_1 + \ldots + f^{(1)}_1 + 3^{n-1} - 2^{n-1}.
 \end{equation}
 
     Let $H$ be a subgraph of $K_{n+1}$ having first Betti number $1$ and a directed path from $v_1$ to $v_{n+1}$. Each such graph must have exactly one vertex of out-degree $2$. If this vertex is $v_i$ for $i \geq 2$, then $H$ is counted by $f^{(1)}_j$ for some $j$ (specifically the index $j$ is the lowest natural number such that $v_j$ is adjacent to $v_1$). Otherwise, it must be the case that the out-degree of $v_1$ is $2$. The number of such subgraphs which are $(1,0, \ldots, 0)$-valid are indexed by words in $0,1$ and $2$; the letter in position $i$ represents how many paths from $v_1$ to $v_n$ intersect in the interval between $v_{i+1}$ and $v_{i+2}$. As there are $3^{n-1}$ words,  there are this many valid subgraphs having out-degree of $v_1$ being 2. However, any word using only the letters $0$ and $2$ represents a graph where both paths used to make the word intersect at every edge;  in other words, such a word corresponds to a subgraph having first Betti number equal to 0. Hence there are $3^{n-1} - 2^{n-1}$ valid subgraphs $H$ that have $\beta_1(H) = 1$ and such that the out-degree of $v_1 = 2$.

      Next, if we let $S(n) := f_{1}^{(n)} + \ldots + f_1^{(1)}$, then the recurrence above gives us the relation $S(n) - S(n-1) = S(n-1) + 3^{n-1} - 2^{n-1}$. Solving the recurrence gives the closed form expression Equation~\eqref{eq::closed_form_dim_1_faces}.
 \end{proof}

The direct counting techniques of this section may also be used to provide a proof of a conjecture of Morales \cite{Morales_private_communication_2024} for the number of vertices of $\Flow_n({\bf a})$.

\begin{Prop}
\label{prop::number_of_vertices}
For ${\bf a}=(a_1,a_2,\ldots,a_n,-\sum_i a_i)$, with $a_i \geq 0$ and signature $sgn({\bf a})=(c_1,c_2,\ldots,c_k)$ then the number $v({\bf a})$ of vertices of $\Flow_n({\bf a})$ equals

\begin{equation}
\label{eq::number_of_vertices}
v({\bf a}) = k^{c_k - 1}\prod_{i =  1}^{k-1} (i+1)^{c_i}.
\end{equation}
\end{Prop}

\begin{proof}

Let ${\bf a}$ be fixed, and let $(c_1,c_2,\ldots,c_k)$ be the signature of ${\bf a}$. We will denote by $V_{{\bf a}}$ the set of vertices of $\Flow_n({\bf a})$ and by $S_{(c_1, \ldots, c_k)}$ the set of tuples of integers ${\bf s} = (s_1, \ldots, s_{n-1})$ where $0 \leq s_i \leq k$ if $s_i$ is in the $k$th \emph{block} determined by $(c_1,c_2,\ldots,c_k)$ (that is, if $c_1 + \ldots c_{k-1} < i \leq c_1 + \ldots c_{k}$).  We will set up a bijection $h: V_{{\bf a}} \rightarrow S_{(c_1, \ldots, c_k)}$ from which Equation~\eqref{eq::number_of_vertices}  then follows.

To define $h$, let $f$ be a flow on $K_{n+1}$ corresponding to a vertex of $\Flow_n({\bf a})$. Since $f$ corresponds to a vertex, the flow entering any vertex leaves along a single edge. Hence for each vertex $v_i \in \{v_2, \ldots ,v_n\}$ (excluding vertices $v_1$ and $v_{n+1}$) the interval $[v_i, v_{i' - 1}]$ is well-defined, where $v_{i'}$ is the unique vertex receiving the flow leaving $v_i$ (note that it is possible that $[v_i, v_{i' - 1}]$ may consist only of the vertex $v_i$ ). Let $s_{i-1}$ be the total incoming flow in the interval $[v_i, v_{i' - 1}]$ excluding the flow at each vertex from the netflow vector. It is easy to verify that $h(f) := (s_1, \ldots, s_{n-1})$ gives the desired bijection.  
\end{proof}

\begin{Ex}
Consider the graph in Figure~\ref{fig::vertex_enumeration_example} representing a vertex of $\Flow_n((1,1,1,0))$.  The flow leaving vertex $v_2$ forms the interval $[v_2, v_3]$, which has a single unit of flow entering from the left. The flow leaving vertex $v_3$ forms the interval $[v_3, v_4]$, which has two units of flow entering from the left. Finally, vertex $v_4$ is the only vertex in its interval. There is only one unit of flow entering $v_4$ from the left. Therefore, the bijection maps the graph of Figure~\ref{fig::vertex_enumeration_example} to the tuple $(1,2,1)$.  
    \begin{figure}
        \centering
        \includegraphics[width=0.3\linewidth]{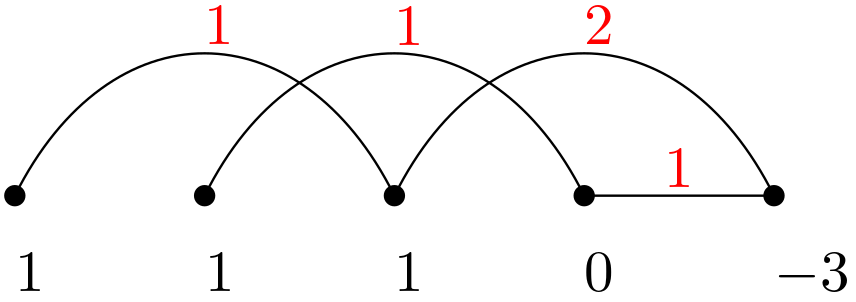}
        \caption{A graph representing a vertex of $\Flow_n((1,1,1,0))$. Under the bijection of Proposition~\ref{prop::number_of_vertices}, the graph corresponds to the tuple $(1,2,1)$.}
        \label{fig::vertex_enumeration_example}
    \end{figure}
\end{Ex}

We remark that Proposition~\ref{prop::number_of_vertices} recovers Proposition~\ref{prop::Chan_Robbins_Yuen_original_number_vertices} in the case that ${\bf a} = (1,0,\ldots, 0)$. Indeed, in this case the signature of ${\bf a}$ is just $(n)$, and so the bijection of Proposition~\ref{prop::number_of_vertices} reduces to returning the tuple of in-degrees of vertices $v_2, \ldots, v_n$.

 This method of direct enumeration stops working as well for higher dimensional faces, with the next case $d = 2$ seemingly not producing a $c$-finite sequence experimentally using SageMath \cite{sagemath, Sage-Combinat}. Nevertheless, we may use similar ideas to produce an alternate formula for the $f$-vector of $CRY_n$ to those in Section~\ref{sect::formulas_for_CRY_n} in the case of faces with low codimension. The following formulas are obtained by starting with all of $K_{n+1}$ and removing edges to lower the first Betti number, while discarding any edge deletion that results in a graph that is not $(1,0,\ldots, 0)$-valid. The benefit to this approach is that the formulas given allow us to calculate individual elements of the $f$-vector of $CRY_n$ without computing the whole $f$-vector first (c.f. Theorem~\ref{thm::main_result}).  

Let $In(v_i)$ be the set of incoming edges adjacent to vertex $v_i$ and $Out(v_i)$ be the set of outgoing edges. In similar fashion, we will denote the cardinalities of these sets as $ind(v_i) := |In(v_i)|$ and $outd(v_i) := |Out(v_i)|$.

\begin{Lem}
\label{lem::technical_lemma_on_primitive_faces}
A subgraph $H \subseteq K_{n+1}$ corresponds to a primitive face of $CRY_n$ if and only if   $In(v_i) \not \subseteq E(H)^C$ for any $i \geq 2$ and  $Out(v_i) \not \subseteq E(H)^C$ for any $i \leq n$. 
    
\end{Lem}

\begin{proof}
    By Hille's result, Theorem~\ref{thm::Hille},  a subgraph $H \subseteq K_{n+1}$ corresponds to a face of $CRY_n$ if and only if it is the support of some valid flow. If $In(v_i) \subseteq E(H)^C$ for some $i \geq 2$, then in particular there is no incoming flow to vertex $v_i$, hence $H$ is either not the support of a valid flow, or is not primitive. Similarly, if $Out(v_i) \subseteq E(H)^C$ for some $i \leq n$, then there is no path from vertex $v_i$ to the sink $v_{n}$. Hence either $H$ is not the support of a valid flow, or is not primitive. 

    Conversely, suppose that $H \subseteq K_{n+1}$ is a subgraph which does  not correspond to a primitive face of $CRY_n$. If $H$ does not correspond to a face, then it must not be the support of a valid flow. The only way for this to happen is if there is some vertex $v_i$ for $i \geq 2$ that is a source or some vertex $v_j$ for $j \leq n$ such that $n$ is a sink. In the former case, we obtain $In(v_i) \subseteq E(H)^C$ and in the latter case we obtain $Out(v_j) \subseteq E(H)^C$.  Finally, we consider the case in which $H$ corresponds to a face of $CRY_n$ which is not primitive. In this case, $H$ arises as the support of some valid flow on $K_{n+1}$, however there exists some vertex $v_i$ which is not in the vertex set of $H$. In this case we have that both $In(v_i) \subseteq E(H)^C$ and $Out(v_i) \subseteq E(H)^C$ . This finishes the proof.
\end{proof}

As mentioned above, we can use this lemma to obtain an efficient formula for the face numbers of $CRY_n$ having low codimension. We will do this by relating sets of edges to a special kind of integer partition defined as follows.

\begin{Def}
    For $n \in \NN$, a \defn{bicolored integer partition of $n$ with distinct bicolored parts} is an integer partition $\lambda \vdash n$ such that the letters of $\lambda$ come in two colors, and the letters appearing in a given color are distinct from one another. We will denote the set of such partitions for all integers $n$ less than some fixed $d$ by $B_{\leq d}$:
     \begin{equation*}
        \defn{B_{\leq d}} := \{ (\mu_1, \mu_2) \, | \, \mu_1 \vdash  k_1, \mu_2 \vdash k_2,  \text{ for all $k_1 + k_2 \leq d$}, \text{ and $\mu_1$ and $\mu_2$ have distinct parts} \}.
    \end{equation*}
\end{Def}

In the following result, we let  $a_i = [x^i] \prod_{k\geq1} (1 - x^k)^2$, the number of partitions of $n$ into an even number of distinct, bicolored parts minus the number of partitions of $n$ into an odd number of distinct, bicolored parts;  see \cite[\href{https://oeis.org/A002107}{A002107}]{oeis}.

 \begin{Thm}
 \label{cor::CRY_faces_low_codim}
     For $1 \leq d \leq n-1$, the entries of the $f$-vector and primitive $f$-vector of $CRY_n$, $f_{\binom{n}{2} - d}^{(n)}$ and $\widetilde{f}_{\binom{n}{2} - d}^{(n)}$, agree and are equal to:
     $$f^{(n)}_{\binom{n}{2} - d} = \widetilde{f}^{(n)}_{\binom{n}{2} - d} =  \sum_{i = 0}^d a_i \binom{\binom{n}{2} - i}{d - i}.$$
 \end{Thm}

 \begin{proof}
     The strategy for this proof is to use inclusion-exclusion on nested collections of graphs, but in a different way than the approach of Section~\ref{sect::main_results}. The idea is to start with the entire complete graph $K_{n+1}$ and obtain subgraphs by deleting edges, while keeping track of edge deletions that lead to graphs which do not correspond to valid faces of $CRY_n$. 

     First we will find a formula for $\widetilde{f}_{\binom{n}{2} - d}^{(n)}$. For $n$ fixed and $d \leq n -1$, observe that any subgraph corresponding to a primitive codimension $d$ face of $CRY_n$ will have $\binom{n}{2} - d$ edges, as for $d$ in this range the removal of each of the $d$ edges is guaranteed to reduce the first Betti number by 1. However, some choices of edges to remove will lead to subgraphs which are not $(1, 0, \ldots , 0)$-valid.
     
     Now let $L^{(d)}_1, \ldots, L^{(d)}_d$ be collections of subgraphs of $K_{n+1}$ where $L^{(d)}_i$ contains those graphs having edge set of size $N - d$ and which do not contain the edges $Out(v_i)$. Similarly, let $R^{(d)}_{\overline{1}}, \ldots , R^{(d)}_{\overline{d}}$ each be a collection of graphs where $R^{(d)}_{\overline{i}}$ contains those graphs having edge set of size $N - d$  and which do not contain the edges $In(v_{n+1 - i})$.  
     Hence by Lemma~\ref{lem::technical_lemma_on_primitive_faces} and the discussion in the previous paragraph, we have that the number of primitive  codimension-$d$ faces of $CRY_n$ is:
     \begin{align*}
         \widetilde{f}_{\binom{n}{2} - d} &= \binom{\binom{n}{2}}{d} - |L_1 \cup \ldots \cup L_d \cup R_{\overline{1}} \cup \ldots \cup R_{\overline{d}}|. 
     \end{align*}
     We can calculate this using the principle of inclusion and exclusion:
     \begin{equation}
     \label{eq::bicolored_partitions_inclusion_exclusion}
         \widetilde{f}_{\binom{n}{2} - d} = \binom{\binom{n}{2}}{d} + \sum_{i = 1  }^d\sum (-1)^i |L_S \cap R_T|,
     \end{equation}
     where the inner sum runs over sets $S\cup T$ such that $S \subseteq \{1, \ldots d \}$, $T \subseteq \{\overline{1}, \ldots , \overline{d}\}$, and such that $\sum_{j \in S\cup T}j = i$ (where in this sum barred elements are treated as normal integers). We may associate to the pair $(L_S, R_T)$ the bicolored integer partition $(\mu_1, \mu_2)$ where the integers appearing in $\mu_1$ are determined by $S$ and the integers appearing in $\mu_2$ are determined by $T$. The integer partition is bicolored by virtue of the distinction between $L_i$ and $R_j$, and the parts of $\mu_1$ (respectively $\mu_2$) are distinct since $S$ (respectively $T$) is a set and not a multiset. 
     Moreover, we claim that $|L_S \cap R_T|$ may be completely determined in terms of $\lambda = (\mu_1, \mu_2)$. Indeed, by the requirement that $d \leq n-1$, the edges excluded from graphs in $L_S$ will be disjoint from those excluded from $R_T$, for otherwise there would exist some $L^{(d)}_i$ and $R^{(d)}_j$ for which $i + j = n$ (contradicting the fact that $d \leq n-1$). Hence the total number of graphs in $|L_S \cap R_T|$ can be found by choosing $d - |\lambda|$ edges from $\binom{n}{2} - |\lambda|$ edges (where for the bicolored partition $\lambda$  we use $|\lambda| := |\mu_1| + |\mu_2|$).
     
     Hence Equation~\eqref{eq::bicolored_partitions_inclusion_exclusion} may be rewritten as a sum over the elements of $B_{\leq d}$ as follows:
     \begin{equation*}
         \widetilde{f}_{\binom{n}{2} - d} = \binom{\binom{n}{2}}{d} + \sum_{\lambda \in B_{\leq d}} (-1)^{|\lambda|} \binom{\binom{n}{2} - |\lambda|}{d - |\lambda|}.
     \end{equation*}
     Collecting terms according to $|\lambda|$ then gives that:
     $$\widetilde{f}^{(n)}_{\binom{n}{2} - d} = \sum_{i = 0}^d a_i \binom{\binom{n}{2} - i}{d - i}.$$
     Finally, it remains to show that $\widetilde{f}^{(n)}_{\binom{n}{2} - d} = f^{(n)}_{\binom{n}{2} - d}$ for $1 \leq d \leq n-1$. This follows from the fact that for every vertex $v_i$ in $K_{n+1}$ we have $ind(v_i) + outd(v_i) = n$. Hence by removing only $d$ edges from $K_{n+1}$, every $(1,0,\ldots, 0)$-valid graph will also be primitive.

 \end{proof}

 Similar techniques can be used to prove a stronger version of Theorem~\ref{cor::CRY_faces_low_codim}; this result is forthcoming.
 
\begin{Ex}
    For $n \geq d + 1$ arbitrary, let $N := \binom{n}{2}$. Then the first few codimensions of the $f$-vector for $CRY_n$ are given by:
    \begin{align*}
        f^{(n)}_{N - 1} & = N - 2\\
        f^{(n)}_{N - 2} & = \binom{N}{2} - 2(N - 1) - 1\\
        f^{(n)}_{N - 3} & = \binom{N}{3} - 2\binom{N - 1}{2} - (N - 2) + 2\\
        f^{(n)}_{N - 4} & = \binom{N}{4} - 2\binom{N - 1}{3} - \binom{N - 2}{2} + 2\binom{N-3}{1} + 1 \\
        f^{(n)}_{N - 4} & = \binom{N}{5} - 2\binom{N - 1}{4} - \binom{N - 2}{3} + 2\binom{N-3}{2} + (N-4) + 2
    \end{align*}
\end{Ex}

\section{Final Remarks}

We conclude with notes on open problems and future work.

\subsection{Face Lattice and Primitive $f$-vectors} The present text deals with finding the $f$-polynomials of $\Flow_{n}({\bf a})$ for any $n \in \NN$ and ${\bf a} \in \NN^n$. A natural question to ask is whether or not these results provide any insights into the face lattice structure of $CRY_n$ and $\Flow_{n}({\bf a})$ more generally.  The identity shown in Equation~\eqref{eq::f_vector_primitive_f_vec_product_result} suggests that there may be some geometric meaning to the primitive $f$-polynomial and to the relationship between $f$ and $\widetilde{f}$. This is current work of the present author. 

\subsection{Combinatorial Proofs} As discussed in Section~\ref{sect::formulas_for_CRY_n}, Theorem~\ref{thm::f_vector_product_result} specializes to a result of Andresen and Kjeldsen \cite[Prop. 4.1]{andresen_kjeldsen_1976} by setting $x = 1$. In this special case, the authors of \cite{andresen_kjeldsen_1976} noted that their proof is algebraic in nature and asked for a bijective proof on the underlying sets of graphs. We remark that our identity, Theorem~\ref{thm::f_vector_product_result}, further refines their identity by providing a grading of the sets according to the first Betti number of the underlying graphs. Hence another direction of current work is to find a bijective proof of   Theorem~\ref{thm::f_vector_product_result}; doing so would simultaneously satisfy the original open problem of Andresen and Kjeldsen. We observe that the  proof strategy proposed by Andresen and Kjeldsen in \cite{andresen_kjeldsen_1976} will not suffice to prove  Theorem~\ref{thm::f_vector_product_result}, as it does not preserve first Betti number in the way required.

\subsection{Other Graphs and Netflows} There are many other open problems under the general theme of extending the results of this manuscript to the setting of other graphs. We conjecture that the approach of this paper may be modified in order to obtain the $f$-polynomials of the \textbf{partition graphs} introduced in \cite{Meszaros2017FlowPO}. It would be interesting to understand the extent to which the techniques of this paper may be used to obtain the $f$-polynomials of flow polytopes for arbitrary graphs and where they fail. Moreover, it would be of further interest to extend this work to the study of $f$-polynomials of $\Flow_n({\bf a})$ while allowing negative entries in ${\bf a}$.

\section*{Acknowledgements}

The author is indebted to Alejandro Morales for proposing this project and for many insightful conversations and feedback, including suggesting the connection to quasisymmetric-like polynomials appearing in Section~\ref{subsect::formulas_as_evaluations_of_sums_of_quasisymmetric_polynomials}. The author further wishes to thank Martha Yip and Rafael Gonz\'alez D'Le\'on for helpful conversations about flow polytopes generally, and LACIM at UQAM for a productive work environment in October 2023 and May 2024.

\printbibliography

@book {EC2,
    AUTHOR = {Stanley, Richard P.},
    TITLE = {Enumerative combinatorics. {V}ol. 2},
    SERIES = {Cambridge Studies in Advanced Mathematics},
    VOLUME = {62},
PUBLISHER = {Cambridge University Press, Cambridge},
     YEAR = {1999},
    PAGES = {xii+581},
}

@online{oeis,
	organization = {OEIS Foundation Inc.},
	label = {OEIS},
	url = {http://oeis.org},
	title = {{The On-Line Encyclopedia of Integer Sequences}}
	}

@article{andresen_kjeldsen_1976,
title = {On certain subgraphs of a complete transitively directed graph},
journal = {Discrete Mathematics},
volume = {14},
number = {2},
pages = {103-119},
year = {1976},
issn = {0012-365X},
doi = {https://doi.org/10.1016/0012-365X(76)90054-6},
url = {https://www.sciencedirect.com/science/article/pii/0012365X76900546},
author = {E. Andresen and K. Kjeldsen}
}

@article{hwang_jin_schlosser_2020,
author = {Hwang, Hsien-Kuei and Jin, Emma and Schlosser, Michael},
year = {2022},
month = {06},
pages = {},
title = {Asymptotics and statistics on Fishburn matrices: Dimension distribution and a conjecture of Stoimenow},
volume = {62},
journal = {Random Structures \& Algorithms},
doi = {10.1002/rsa.21100}
}

@article{Jelinek_2011,
title = {Counting general and self-dual interval orders},
journal = {Journal of Combinatorial Theory, Series A},
volume = {119},
number = {3},
pages = {599-614},
year = {2012},
issn = {0097-3165},
doi = {https://doi.org/10.1016/j.jcta.2011.11.010},
url = {https://www.sciencedirect.com/science/article/pii/S0097316511001853},
author = {Vít Jelínek}
}

@article{zeilberger_proof_1998,
  author  = {Doron Zeilberger},
  title   = {Proof of a conjecture of Chan, Robbins, and Yuen},
  journal = {Electron. Trans. Numer. Anal.},
  volume  = {9},
  year    = {1999},
  pages   = {147--148},
}

@article{Hille_2003,
title = {Quivers, cones and polytopes},
journal = {Linear Algebra and its Applications},
volume = {365},
pages = {215-237},
year = {2003},
note = {Special Issue on Linear Algebra Methods in Representation Theory},
issn = {0024-3795},
doi = {https://doi.org/10.1016/S0024-3795(02)00406-8},
url = {https://www.sciencedirect.com/science/article/pii/S0024379502004068},
author = {Lutz Hille}
}

@article{Chan_Robbins_Yuen_2000,
author = {Clara S. Chan and David P. Robbins and David S. Yuen},
title = {{On the volume of a certain polytope}},
volume = {9},
journal = {Experimental Mathematics},
number = {1},
publisher = {A K Peters, Ltd.},
pages = {91 -- 99},
year = {2000},
}

@article{Stoimenow_1998,
author = {Stoimenow, A.},
title = {Enumeration of chord diagrams and an upper bound for Vassiliev  invariants},
journal = {Journal of Knot Theory and Its Ramifications},
volume = {07},
number = {01},
pages = {93-114},
year = {1998},
doi = {10.1142/S0218216598000073},

URL = { 
    
        https://doi.org/10.1142/S0218216598000073

}
}

@article{Bousquet_Melou_2010,
title = {(2+2)-free posets, ascent sequences and pattern avoiding permutations},
journal = {Journal of Combinatorial Theory, Series A},
volume = {117},
number = {7},
pages = {884-909},
year = {2010},
issn = {0097-3165},
doi = {https://doi.org/10.1016/j.jcta.2009.12.007},
url = {https://www.sciencedirect.com/science/article/pii/S0097316509001885},
author = {Mireille Bousquet-Mélou and Anders Claesson and Mark Dukes and Sergey Kitaev}
}

@article{Meszaros2017FlowPO,
  title={Flow Polytopes of Partitions},
  author={Karola M\'esz\'aros and Connor Simpson and Zoe Wellner},
  journal={Electron. J. Comb.},
  year={2019},
  volume={26},
  pages={1},
  doi={https://doi.org/10.37236/8114}
}

@misc{Bjorner_Stanley_2005,
  title={An Analogue of Young's Lattice for Compositions},
  author={Anders Bj\"orner and Richard P. Stanley},
  EPRINT = {math/0508043},
  EPRINTTYPE = {arxiv},
  year={2005},
  doi={https://doi.org/10.48550/arXiv.math/0508043}
}

@article{Petersen_2005,
author = {Petersen, Kyle},
year = {2005},
month = {09},
pages = {},
title = {A Note on Three Types of Quasisymmetric Functions},
volume = {12},
journal = {Electronic Journal of Combinatorics},
doi = {10.37236/1958}
}

@article{Tesler_polytopes_2015,
  title={The polytope of Tesler matrices},
  author={M{\'e}sz{\'a}ros, Karola and Morales, Alejandro H and Rhoades, Brendon},
  journal={Selecta Mathematica},
  volume={23},
  pages={425--454},
  year={2017},
  publisher={Springer}
}

@misc{Baldoni_Vernge_2001,
  title={Residues formulae for volumes and Ehrhart polynomials of convex polytopes.},
  author={Welleda Baldoni-Silva and Mich{\`e}le Vergne},
  EPRINT = {math/0103097},
  EPRINTTYPE = {arxiv},
  year={2001},
  doi={https://doi.org/10.48550/arXiv.math/0103097}
}

@misc{PS_vertices,
author = {Dugan, William and Hegarty, Maura and Morales, Alejandro and Raymond, Annie},
year = {2023},
month = {07},
EPRINT = {2307.09925},
EPRINTTYPE = {arxiv},
title = {Generalized Pitman-Stanley polytope: vertices and faces}
}

@article{Benedetti_transactions_2018,
  title={A combinatorial model for computing volumes of flow polytopes},
  author={Benedetti, Carolina and Gonz{\'a}lez D’Le{\'o}n, Rafael and Hanusa, Christopher and Harris, Pamela and Khare, Apoorva and Morales, Alejandro and Yip, Martha},
  journal={Transactions of the American Mathematical Society},
  volume={372},
  number={5},
  pages={3369--3404},
  year={2019}
}

@manual{sagemath,
	author = {The Sage Developers},
	key = {SageMath},
	url = {https://www.sagemath.org},
	title = {{S}ageMath, the {S}age {M}athematics {S}oftware {S}ystem ({V}ersion 9.7)},
	label = {Sage}
	}

@Misc{Sage-Combinat,
      Author = {The {S}age-{C}ombinat community},
      label = {S-C},
      Title = {{S}age-{C}ombinat: enhancing {S}age as a toolbox for computer exploration in algebraic combinatorics},
      url= {http://combinat.sagemath.org}}

@misc{Morales_private_communication_2024,
  author = "Alejandro Morales",
  howpublished = "private communication",
  year = "2024",
  month = "8"
}

@article{meszaros_morales_ehrhart_2019,
	title = {Volumes and {Ehrhart} polynomials of flow polytopes},
	volume = {293},
	issn = {1432-1823},
	url = {https://doi.org/10.1007/s00209-019-02283-z},
	doi = {10.1007/s00209-019-02283-z},
	number = {3},
	journal = {Mathematische Zeitschrift},
	author = {Mészáros, Karola and Morales, Alejandro H.},
	month = dec,
	year = {2019},
	pages = {1369--1401},
}


\appendix
\section{Proof of Lemma~\ref{lem::helper_lemma_to_main_result} }\label{sect::appendixA}
\renewcommand{\thesubsection}{\Alph{subsection}}
\numberwithin{Thm}{subsection}
\numberwithin{equation}{subsection}

\begin{proof}[Proof of Lemma~\ref{lem::helper_lemma_to_main_result}]
To prove Equation~\eqref{eq::main_result_proof_multivar_identity}, we will proceed by induction on the length of $\alpha$, demonstrating at each stage that for each composition $\gamma'$, the coefficient of ${\bf x}^{\gamma'}$ on the left-hand side of Equation~\eqref{eq::main_result_proof_multivar_identity}  is the same as that on the right-hand side. 

First, let us suppose that $\alpha$ has a single part, i.e. that $\alpha = (n)$ for some $n \in \NN$, and consider some fixed but arbitrary $\mu \composition n$. Then the compositions $\beta$ appearing on the left-hand side of Equation~\eqref{eq::main_result_proof_multivar_identity} are $(1) , \ldots, (n)$, with $\mu$ appearing in $P_{(n)}, \ldots, P_{(n - \mu_1 + 1)}$. Hence the coefficient of ${\bf x}^{\mu}$ on the left-hand side of Equation~\eqref{eq::main_result_proof_multivar_identity} is $\sum_{i = 0}^{\mu_1 - 1} (-1)^{i + \mu_1 - 1}\binom{n-1}{i}$ which simplifies by a binomial identity to $(-1)^{\mu_1 - 1}\binom{n- 2}{\mu_1 - 1}$. On the other hand, since $\mu \composition n$, we have that:
\begin{align*}
    \binom{n - 2}{\mu_1 - 1} &= \binom{n - 2}{\mu_2 + \ldots +  \mu_{\ell(\mu)} - 1} \\
    &= [y^{n-2}] \frac{y^{\mu_2 + \ldots +  \mu_{\ell(\mu)} - 1}}{(1 - y)^{\mu_2 + \ldots +  \mu_{\ell(\mu)}}} = [y^{n-1}] \frac{ y\cdot y^{\mu_2 - 1}}{(1 - y)^{\mu_2}} \cdots \frac{ y\cdot y^{\mu_{\ell(\mu)} - 1}}{(1 - y)^{\mu_{\ell(\mu)}}},
    \end{align*}
    where in the second equality above we introduce the variable $y$ for the purpose of using techniques of generating functions in our manipulations. Performing the indicated coefficient extraction, we obtain:
    \begin{align*}
       \binom{n - 2}{\mu_1 - 1}  &= \sum_{\substack{c_1 + \ldots + c_{\ell(\mu)} = n-1\\ c_{i} \geq 0}} \binom{c_1 - 1}{\mu_2 - 1} \cdots \binom{c_{\ell(\mu) - 1} - 1}{\mu_{\ell(\mu)} - 1}\\
    &= \sum_{\substack{\beta \composition n \\ \ell(\beta) = \ell(\mu) \\
    \beta_1 = 1}} \binom{\beta_2 - 1}{\mu_2 - 1} \cdots \binom{\beta_{\ell(\mu)} - 1}{\mu_{\ell(\mu)} - 1},
\end{align*}
 which is the coefficient of ${\bf x}^{\mu}$ on the right-hand side of Equation~\eqref{eq::main_result_proof_multivar_identity}.

 Now suppose the statement is true for partitions of $n$ with $k$ parts, and consider the case of $\alpha = \alpha_1\cdots \alpha_{k+1} \composition n$. Consider the coefficient of ${\bf x}^{\mu}$ in the left-hand side of Equation~\eqref{eq::main_result_proof_multivar_identity} for $\mu = \mu_1 \ldots \mu_t \composition n$. Furthermore, consider the prefix composition $\alpha' := \alpha_1\cdots \alpha_{k} \composition n - \alpha_{k+1}$. Starting with the left-hand side of Equation~\eqref{eq::main_result_proof_multivar_identity} we have:
 \begin{multline*}
     \sum_{\beta \leqc \alpha} x_1^{|\alpha| - |\beta|}\left(\prod_{i = 1}^{\ell(\alpha)}\binom{\alpha_i - 1}{\alpha_i - \beta_i}\right)  P_{\beta}(x_1, x_2, \ldots, x_{|\beta|})  \\= \sum_{i = 0}^{\alpha_{k+1}} \binom{\alpha_{k+1}}{i} x_1^i\sum_{\beta \leqc \alpha'} x_1^{|\alpha'| - |\beta|}\left(\prod_{i = 1}^{\ell(\alpha')}\binom{\alpha_i - 1}{\alpha_i - \beta_i}\right)  P_{\beta i}(x_1, x_2, \ldots, x_{|\beta i|}),
 \end{multline*}
 where we use $\beta i$ to mean the composition $\beta$ with letter $i$ concatenated to the end. By applying the inductive hypothesis to the inner sum, we may conclude that the sum of coefficients of terms ${\bf x}^{\mu'}$ where $\mu'$ is the prefix $\mu' := \mu_1 \cdots \mu_{t-1}$ is equal to:
 \begin{multline*}
     \sum_{i = 0}^{\alpha_{k+1} - 1}\binom{\alpha_{k+1} - 1}{i}\sum_{m_1 + \ldots m_{\ell(\alpha) - 1} = \ell(\mu) - 1 - i}\binom{\alpha_1 - 2}{\mu_2 + \ldots + \mu_{m_1 + 1} - 1}\cdot \binom{\alpha_2 - 1}{\mu_{m_1 + 2} + \ldots + \mu_{m_1 + m_2 + 1} - 1} \\ \cdots \binom{\alpha_{\ell(\alpha) - 1} - 1}{\mu_{m_1 + \ldots m_{\ell(\alpha) - 2}} + \ldots + \mu_{m_1 + \ldots m_{\ell(\alpha) - 1}} - 1}.
 \end{multline*}
 The terms from the sum above will contribute to the coefficient of ${\bf x}^{\mu}$ if and only if $\mu_{m_1 + \ldots m_{\ell(\alpha) - 1}} + \ldots + \mu_{m_1 + \ldots m_{\ell(\alpha) }} - 1 = i$. Hence, the coefficient of ${\bf x}^{\mu}$ in the left-hand side of Equation~\eqref{eq::main_result_proof_multivar_identity} is equal to:
 \begin{align*}
     \sum_{m_1 + \ldots +m_{\ell(\alpha)} = \ell(\mu) - 1}\binom{\alpha_1 - 2}{\mu_2 + \ldots + \mu_{m_1 + 1} - 1}\cdot \binom{\alpha_2 - 1}{\mu_{m_1 + 2} + \ldots + \mu_{m_1 + m_2 + 1} - 1} \cdots \binom{\alpha_{\ell(\alpha)} - 1}{\mu_{m_1 + \ldots m_{\ell(\alpha) - 1}} + \ldots + \mu_{m_1 + \ldots m_{\ell(\alpha) }} - 1}.
 \end{align*}
  On the other hand,  we may think of each binomial coefficient in the product of each term above as the coefficient extraction of a generating function:
\begin{equation*}
    \binom{\alpha_{\ell(\alpha) - 1} - 1}{\mu_{m_1 + \ldots m_{\ell(\alpha) - 2}} + \ldots + \mu_{m_1 + \ldots m_{\ell(\alpha) - 1}} - 1} = [y^{\alpha_{\ell(\alpha)} - 1}]\frac{y^{\mu_{m_1 + \ldots m_{\ell(\alpha) - 1}} + \ldots + \mu_{m_1 + \ldots m_{\ell(\alpha) }} - 1 }}{(1 - y)^{\mu_{m_1 + \ldots + m_{\ell(\alpha) - 1}} + \ldots + \mu_{m_1 + \ldots +m_{\ell(\alpha) }}}} .
\end{equation*}
  After making this substitution, the above sum becomes:
 \begin{align*}
     &\sum_{m_1 + \ldots + m_{\ell(\alpha)} = \ell(\mu) - 1} [y^{\alpha_1 - 1}]\frac{y\cdot y^{\mu_2 - 1}}{(1 - y)^{\mu_2}} \cdots \frac{y\cdot y^{\mu_{m_1 + 1} - 1}}{(1 - y)^{\mu_{m_1 + 1}}}    \cdots  [y^{\alpha_{\ell(\alpha)} }]  \frac{y\cdot y^{\mu_{m_1 + \ldots m_{\ell(\alpha) - 1}}}}{(1 - y)^{\mu_{m_1 + \ldots m_{\ell(\alpha) - 1}}}} \cdots \frac{y\cdot y^{\mu_{m_1 + \ldots +m_{\ell(\alpha)}}}}{(1-y)^{\mu_{m_1 + \ldots +m_{\ell(\alpha)}}}}.
\end{align*}
We make special note of the fact that the coefficient of the first index in each term is $[y^{\alpha_1 - 1}]$, whereas the coefficient of every other index is of the form $[y^{\alpha_i}]$. After manipulation, the above sum simplifies to: 
$$\sum_{\substack{\beta \geqa \alpha \\ \ell(\beta) = \ell(\mu) \\ \beta_1 = 1}} \prod_{i = 2}^{\ell(\beta)} [y^{\beta_i - 1}] \frac{y^{\mu_i - 1}}{(1 - y)^{\mu_i}},$$ 
which is the coefficient of ${\bf x}^{\mu}$ in the right-hand side of Equation~\eqref{eq::main_result_proof_multivar_identity}.

\end{proof}

\end{document}